\documentclass[11pt]{article}
\usepackage{amsmath,amssymb,color,graphicx,bbm,amsthm,verbatim}
\usepackage{tikz}
\usepackage{rotating}
\usetikzlibrary{arrows,patterns}
\usetikzlibrary{quotes,angles}
\usepackage{bbold}
\usepackage{algpseudocode}
\usepackage[plain]{algorithm} 
\usepackage{mleftright}
\mleftright 
\usepackage{soul}
\usepackage{cancel}
\usepackage{caption}
\newcommand {\R}{\mathbb{R}}

\newcommand{\sgn}{\mathrm{sgn}}
\newcommand {\grad}{\nabla}

\newcommand{\matlab}{{\sc Matlab}}

\newcommand{\magenta}[1]{\textcolor{magenta}{#1}}

\newcommand{\beq}{\begin{equation}}
\newcommand{\eeq}{\end{equation}}
\newcommand{\beqf}{\begin{flalign}}
\newcommand{\eeqf}{\end{flalign}}

\newcommand{\RA}{\Rightarrow}

\newtheorem{theorem}                {Theorem}
\newtheorem{corollary}  [theorem]   {Corollary}
\newtheorem{lemma}      [theorem]   {Lemma}

\newcommand\numberthis{\addtocounter{equation}{1}\tag{\theequation}}
\makeatletter

\newlength{\hatchspread}
\newlength{\hatchthickness}
\newlength{\hatchshift}
\newcommand{\hatchcolor}{}
\tikzset{hatchspread/.code={\setlength{\hatchspread}{#1}},
	hatchthickness/.code={\setlength{\hatchthickness}{#1}},
	hatchshift/.code={\setlength{\hatchshift}{#1}},
	hatchcolor/.code={\renewcommand{\hatchcolor}{#1}}}
\tikzset{hatchspread=3pt,
	hatchthickness=0.4pt,
	hatchshift=0pt,
	hatchcolor=black}
\pgfdeclarepatternformonly[\hatchspread,\hatchthickness,\hatchshift,\hatchcolor]
{custom north west lines}
{\pgfqpoint{\dimexpr-2\hatchthickness}{\dimexpr-2\hatchthickness}}
{\pgfqpoint{\dimexpr\hatchspread+2\hatchthickness}{\dimexpr\hatchspread+2\hatchthickness}}
{\pgfqpoint{\dimexpr\hatchspread}{\dimexpr\hatchspread}}
{
	\pgfsetlinewidth{\hatchthickness}
	\pgfpathmoveto{\pgfqpoint{0pt}{\dimexpr\hatchspread+\hatchshift}}
	\pgfpathlineto{\pgfqpoint{\dimexpr\hatchspread+0.15pt+\hatchshift}{-0.15pt}}
	\ifdim \hatchshift > 0pt
	\pgfpathmoveto{\pgfqpoint{0pt}{\hatchshift}}
	\pgfpathlineto{\pgfqpoint{\dimexpr0.15pt+\hatchshift}{-0.15pt}}
	\fi
	\pgfsetstrokecolor{\hatchcolor}
	\pgfusepath{stroke}
}

\makeatother
\title{Analysis of Limited-Memory BFGS on a Class of Nonsmooth Convex Functions}
\author{Azam Asl\thanks{Courant Institute of Mathematical Sciences, New York University.
		Supported by a grant from the Simons Foundation (417314,MHW).}
	\and 
	Michael L.~Overton\thanks{Courant Institute of Mathematical Sciences, New York University. Supported in
		part by National Science Foundation Grant DMS-1620083.}
}
\date{\today}
\begin{document}
	\maketitle

\begin{abstract}
The limited memory BFGS (L-BFGS) method is widely used for large-scale unconstrained optimization, but its behavior on nonsmooth problems has received little attention. L-BFGS can be used with or without ``scaling"; the use of scaling is normally recommended. A simple special case, when just one BFGS update is stored and used at every iteration, is sometimes also known as memoryless BFGS. We analyze memoryless BFGS with scaling, using any Armijo-Wolfe line search, on the function $f(x) = a|x^{(1)}| + \sum_{i=2}^{n} x^{(i)}$, initiated at any point $x_0$ with $x_0^{(1)}\not = 0$. We show that if $a\ge 2\sqrt{n-1}$, the absolute value of the normalized search direction generated by this method converges to a constant vector, and if, in addition, $a$ is larger than a quantity that depends on the Armijo parameter, then the iterates converge to a non-optimal point $\bar x$ with $\bar x^{(1)}=0$, although $f$ is unbounded below. As we showed in previous work, the gradient method with any Armijo-Wolfe line search also fails on the same function if $a\geq \sqrt{n-1}$ and $a$ is larger than another quantity depending on the Armijo parameter, but scaled memoryless BFGS fails under a \emph{weaker} condition relating $a$ to the Armijo parameter than that implying failure of the gradient method. Furthermore, in sharp contrast to the gradient method, if a specific standard Armijo-Wolfe bracketing line search is used, scaled memoryless BFGS fails when $a\ge 2 \sqrt{n-1}$ \emph{regardless} of the Armijo parameter. Finally, numerical experiments indicate that the results may 
extend to scaled L-BFGS with any fixed number of updates $m$, and to more general piecewise linear functions.
\end{abstract}

\section{Introduction}\label{sec:intro}
The limited memory BFGS (L-BFGS) method is widely used for large-scale unconstrained optimization, but its behavior on nonsmooth problems has received little attention. In this paper we give the first analysis of an instance of the method, sometimes known as memoryless BFGS with scaling,
 on a specific class of nonsmooth convex problems, showing that under given conditions the method generates iterates whose function values are
 bounded below, although the function itself is unbounded below.

The ``full" BFGS method \cite[Sec.~6.1]{NW06},  independently derived by Broyden, Fletcher, Goldfarb and Shanno in 1970, is remarkably effective for unconstrained
optimization, 
but even when the minimization objective $f:\R^n\rightarrow \R$ is assumed to be twice continuously differentiable and convex,
with bounded level sets, the analysis of the method is nontrivial. 
Powell \cite{POW76b} gave the first convergence analysis for full BFGS using an Armijo-Wolfe line search for this class of functions,
establishing convergence to the minimal function value.
In the smooth, nonconvex case it is generally accepted that the method is very reliable for finding stationary points (usually local minimizers), although pathological counterexamples exist \cite{DAI02,MAS04}. 

At first glance, it might appear that, since BFGS uses gradient differences to approximiate information about the Hessian of $f$,
the use of BFGS for nonsmooth optimization makes little sense: first, because at minimizers where $f$
is not differentiable, neither the gradient nor the Hessian exists; and secondly, even at other points where $f$ is twice differentiable,
the Hessian might appear to be meaningless: for example, for a piecewise linear function such as studied in this paper, the Hessian
is zero everywhere that it is defined.  However, the way to make sense of the applicability of BFGS to a nonsmooth function is to consider
its approximation by a very ill-conditioned smooth function. For example, the function $f(x)=\|x\|_2$ can be 
arbitrarily well approximated by the smooth function $f(x)=\sqrt{\|x\|_2^2 + \epsilon^2}$, where $\epsilon>0$. As $\epsilon\downarrow 0$, the approximation
becomes arbitrarily good --- but also arbitrarily ill-conditioned. For any \emph{fixed} $\epsilon>0$, the BFGS convergence theory applies. As $\epsilon \downarrow 0$, it is not at all clear what impact the property of good approximation via badly conditioned functions has on the convergence theory,
which, of course, does not apply when $\epsilon=0$.
Nonetheless, even for $\epsilon=0$, the method remains well defined, as the gradient is defined everywhere
except at the minimizer (the origin). In fact, it was established recently by 
Guo and Lewis \cite{GL18} that Powell's result for smooth functions mentioned above can be extended, in a nontrivial way, to
show that the iterates generated by BFGS with an Armijo-Wolfe line search, when applied to $f(x)=\|x\|_2$,
converge to the origin. Even the case $n=1$, where $f$ is the absolute value function, is surprisingly complex; 
it turns out that in this case the sequence of iterates is defined by a certain binary expansion of the starting point \cite{LO13}. 
However, in this simple example it is easy to see intuitively \emph{why} BFGS works well. The line search ensures that
the iterates oscillate back and forth across the origin, giving a gradient difference equal to 2 at every iteration. As the iterates
converge to the origin, the result is that the ``inverse Hessian approximation'' generated by BFGS converges to zero, 
resulting in quasi-Newton steps that also converge to zero. An important consequence is that the line search
never requires many function evaluations. In contrast, when gradient descent with the same line search is applied to the absolute value function,
the iterates converge to the origin, but each line search requires a number of function evaluations that increases with the iteration number.

More generally, if $f$ is locally Lipschitz, BFGS is still typically
well defined, because such functions are differentiable almost everywhere by Rademacher's theorem \cite{Cla90},
and hence $f$ is differentiable at a randomly generated point with probability one.  Furthermore, substantial computational experience \cite{LO13}
shows that even when $f$ is nonsmooth, the method is remarkably reliable for finding Clarke stationary points (again, typically
local minimizers), and furthermore, this property extends in a certain sense to constrained problems \cite{CurMitOve17}.
Indeed, no non-pathological counterexamples showing convergence to non-stationary values, 
meaning in particular examples where the starting point is not predetermined but generated randomly, are known.
The superlinear convergence rate that holds generically for smooth functions is not attained in the nonsmooth case;
instead, full BFGS is observed to converge linearly, in a sense described in \cite{LO13}, on nonsmooth functions.
Furthermore, in general one does not observe the inverse Hessian approximation converging to zero;
instead, what seems to be typical is that \emph{some} of its eigenvalues converge to zero, with corresponding eigenvectors identifying directions along which
$f$ is nonsmooth at the minimizer. See \cite[Sec. 6.2]{LO13} for details.

The full BFGS method maintains and updates an approximation to the inverse (or a factorization) of the Hessian matrix $\grad^2 f(x)$
at every iteration, defined by current known gradient difference information $y_{k-1}=\grad f(x_k)-\grad f(x_{k-1})$ along with  $s_{k-1}=x_{k}-x_{k-1}$.
The use of the Wolfe condition in the line search, requiring an increase in the directional derivative of $f$ along the descent direction
generated by BFGS, ensures that the updated inverse Hessian approximation is positive definite.
Since the update has rank two, the cost of full BFGS is $O(n^2)$ operations per iteration.
While this was a great advance over the cost of Newton's method in the 1970s, already in the 1980s it was realized that the cost was too
high for problems where $n$ is large, 
and hence the limited memory version, L-BFGS, became popular, and is widely used today
(see \cite{AM15,DL11,HarchaouiXX}, for example).
 The standard version of L-BFGS was introduced by Liu and Nocedal \cite{LN89} and is also discussed in detail in \cite[Sec.~7.2]{NW06}.
Let $m \ll n$ be given.
Instead of maintaining an approximation to the inverse Hessian, at the $k$th iteration a proxy for this matrix is 
implicitly defined by application of
the most recent $m$ BFGS updates (which are defined by saving $y_j$ and $s_j$ from the past $m$ iterations)
to a given sparse matrix $H_k^0$. One possible choice for $H_k^0$ is the identity matrix $I$,
but a popular choice is to instead use \emph{scaling}, defining 
\beq
               H_k^0 = \frac{s_{k-1}^T y_{k-1}}{y_{k-1}^T y_{k-1}} I.    \label{Hk0def}
 \eeq

Analysis of L-BFGS is more straightforward than analysis of full BFGS in the case that $f$ is smooth and strongly convex,
and is given in \cite[Theorem 7.1]{LN89}, where linear convergence to minimizers is established, regardless of whether scaling is used
or not. Furthermore, it is stated in \cite{LN89} 
that scaling greatly accelerates L-BFGS, and this seems to be the currently accepted wisdom. However, we
show in this paper that it is exactly the choice of scaling that may result in failure of L-BFGS on a specific class of nonsmooth functions. This situation is
in sharp contrast to our experience with full BFGS on nonsmooth functions, where the same algorithm that is normally used for smooth
functions works well also on nonsmooth functions. 

We consider the convex function
\beq \label{fdef}
         f(x) = a|x^{(1)}| + \sum_{i=2}^{n} x^{(i)},
\eeq
where $a \geq \sqrt{n-1}$.
Note that although $f$ is unbounded below, it is bounded below along the line defined by the negative gradient direction from any
point $x$ with $x^{(1)}\not = 0$. In \cite{AO18}
we analyzed the gradient method with \emph{any} Armijo-Wolfe line search applied to \eqref{fdef}. 
We showed that
if 
	\beq \label{gradient-armijo_cond}
	a > \sqrt{(\frac{1}{c_1} -1)(n-1)},
	\eeq
where $c_1$ is the Armijo parameter, the gradient method, initiated at \emph{any} point
$x_0$ with $x_0^{(1)}\not = 0$, fails in the sense
that it generates a sequence converging to a non-optimal point $\bar x$ with $\bar x^{(1)}=0$, although $f$ is unbounded below. In the
present paper, we analyze
scaled L-BFGS with $m=1$, i.e., with just one update --- a method sometimes known as \emph{memoryless} BFGS \cite[p.~180]{NW06}
--- applied to the function \eqref{fdef}, and identify conditions under which the method converges to  non-optimal points (more details are given in the next paragraph).
In contrast, it is known that when \emph{full} BFGS is applied to the same function,
eventually the method generates a search direction on which $f$ is unbounded below \cite{XW17}; see also 
\cite{LZ15}.
The specific choice of objective function $f$ offers two advantages: one is its simplicity, but another is that there is little difficulty distinguishing in practice whether the method ``succeeds'' or ``fails'' from a given starting point: success is associated with a sequence of function values that is unbounded below, while convergence of the sequence to a finite value implies failure.

The paper is organized as follows. In \S \ref{sec:memBFGS}, we define the scaled memoryless BFGS method, using any
line search satisfying the Armijo and Wolfe conditions, and derive
some properties of the method applied to the function $f$ in \eqref{fdef}, initiated at any point $x_0$ with $x_0^{(1)}\not = 0$. In \S\ref{subsec:AWexist}, we show that if $a\geq \sqrt{3(n-1)}$, the algorithm is well defined in the sense
that Armijo-Wolfe steplengths always exist, deferring the technical details to Appendix~\ref{appendA}. Then in
\S \ref{sec:theory}, we give our main theoretical results. First, in \S 
\ref{subsec:maincondition}, we show that  if $a\ge 2\sqrt{n-1}$, in the limit the absolute value of the normalized search direction 
generated by the method converges to a constant vector, deferring the most technical parts of the proof to 
Appendix~\ref{appendB}.
Then, in \S \ref{subsec:armijo}, we show that if
$a$ further satisfies a condition depending on the Armijo parameter, the method converges
to a non-optimal point $\bar x$ with $\bar x^{(1)} = 0$. 
Furthermore, this condition is \emph{weaker} than the corresponding condition \eqref{gradient-armijo_cond} for the
gradient method. Then, in \S \ref{subsec:specificLS},
we show that, if a specific standard Armijo-Wolfe bracketing line search is used,
scaled memoryless BFGS converges to a non-optimal point
when $a\ge 2\sqrt{n-1}$ \emph{regardless} of the Armijo parameter. This is in sharp contrast 
to the gradient method using the same line search, for which success or failure on the function $f$ depends 
on the Armijo parameter.
In \S \ref{sec:expts} we present
 some numerical experiments which support our theoretical results, and which indicate that the results may
extend to scaled L-BFGS with any fixed number of updates $m$, and to more general piecewise linear functions. We make some concluding remarks in 
\S \ref{sec:conclude}.


\section{The Memoryless BFGS Method}\label{sec:memBFGS}
First let $f$ denote any locally Lipschitz function mapping $\R^n$ to $\R$, and let $x_{k-1} \in \R^n$ denote the $(k-1)$th iterate of an optimization algorithm
where $f$ is differentiable at $x_{k-1}$ with gradient $\grad f(x_{k-1})$. 
Let $d_{k-1}\in\R^n$
denote a descent direction, i.e., satisfying $\grad f(x_{k-1})^T d_{k-1} < 0$.
Let parameters $c_1$ and $c_2$, known as the Armijo and Wolfe parameters, satisfy  $0<c_1<c_2< 1 $.
We say that the steplength $t$ satisfies the Armijo condition at iteration ${k-1}$ if 
\beq
f(x_{k-1} + t d_{k-1}) \leq f(x_{k-1}) + c_1 t \grad f(x_{k-1})^T d_{k-1} \label{armijo_cond}
\eeq
and that it satisfies the Wolfe condition if 
\beq
\nabla f(x_{k-1} + t d_{k-1}) \text{ exists with }  \grad f(x_{k-1} + t d_{k-1})^T d_{k-1} \geq c_2 \grad f(x_{k-1})^T d_{k-1}. \label{wolfe_cond}
\eeq
It is known that if $f$ is smooth or convex, and bounded below along the direction $d_{k-1}$,
a point satisfying these conditions must exist (see \cite[Theorem 4.5]{LO13} for weaker conditions on $f$ for which this holds). 
Note that as long as $f$ is differentiable at the initial iterate, defining subsequent
iterates by $x_{k}=x_{k-1}+t_{k-1} d_{k-1}$, where \eqref{wolfe_cond} holds for $t=t_{k-1}$, ensures that $f$ is differentiable at $x_k$.

We are now ready to define the memoryless BFGS method (L-BFGS with $m=1$), also known as L-BFGS-1, with scaling, i.e.,
with $H_k^0$ defined by \eqref{Hk0def}. The algorithm is defined for any $f$, but its analysis will be specifically for \eqref{fdef}.
\begin{flalign*}
&\mathbf{Algorithm~1~(Memoryless~BFGS~with~scaling),~with~input~x_0} &&\\
&~~~d_0= -\nabla f(x_0)  \numberthis\label{d0}&&\\
&~~~\mathbf{for~~} k=1,2,3,\ldots, \mathbf{define} &&\\
&~~~~~~~~~~t_{k-1} = t\mathrm{~satisfying~\eqref{armijo_cond}~and~ \eqref{wolfe_cond}}&&\\
&~~~~~~~~~~x_k=x_{k-1} +t_{k-1}d_{k-1}&& \numberthis \label{update}&&\\
&~~~~~~~~~~s_{k-1}  = x_k - x_{k-1}   \numberthis\label{sk}&&\\
&~~~~~~~~~~y_{k-1} = \nabla f(x_{k}) - \nabla f(x_{k-1}) \numberthis \label{yk}&&\\
&~~~~~~~~~~V_{k-1} = I -  \dfrac{y_{k-1}s_{k-1}^T}{y^T_{k-1}s_{k-1}} \numberthis\label{Vk} &&\\
&~~~~~~~~~~H_k = \dfrac{s_{k-1}^Ty_{k-1}}{y^T_{k-1}y_{k-1}}V_{k-1}^TV_{k-1} +\dfrac{s_{k-1}s^T_{k-1}}{s^T_{k-1}y_{k-1}} \numberthis\label{hknw} &&\\
&~~~~~~~~~~d_k=-H_k\grad f(x_k) \numberthis\label{dir}&&\\
&~~~\mathbf{end~~}&&\\
\end{flalign*} 
Let us adopt the convention that if no steplength $t$ exists satisfying the Armijo and Wolfe conditions \eqref{armijo_cond} and \eqref{wolfe_cond}, the algorithm is terminated.  Hence, for any smooth or convex function, termination implies that a direction $d_{k-1}$ has been
identified along which $f(x_{k-1}+t d_{k-1})$ is unbounded below.

Now let us restrict our attention to the convex function $f$ given in \eqref{fdef}.  The question we address in this paper
is whether memoryless BFGS will succeed in identifying the fact that $f$ is unbounded below, either because it generates a direction $d$ for
which no steplength $t$ satisfying the Armijo and Wolfe conditions exists (in which case the algorithm terminates), or, alternatively, that it generates a 
sequence $\{x_k\}$ for which Armijo-Wolfe steps always exist, with $f(x_k)\downarrow -\infty$.  If neither event takes place, $\{f(x_k)\}$ is bounded below,
which is regarded as failure, since $f$ is unbounded below.

For the function \eqref{fdef}, requiring $t_{k-1}$ to satisfy \eqref{wolfe_cond}, regardless of the value of the Wolfe parameter $c_2\in (0,1)$,
is, via \eqref{update}, equivalent to the condition
\beq \sgn(x^{(1)}_{k}) = -\sgn(x^{(1)}_{k-1}). \label{sgn}\eeq
Via \eqref{sk} we see that \eqref{sgn} is equivalent to the condition
\beq\label{sgnxk_1}
|s^{(1)}_{k-1}| = |x^{(1)}_{k-1}| + |x^{(1)}_k|.
\eeq
Without loss of generality, we assume that the initial point $x_0$ has a positive first component, i.e., \mbox{$x_0^{(1)} > 0$}, so that
\beq \nabla f(x_{k}) =  \left[\begin{array}{c} (-1)^k a  \\ \mathbb{1}\end{array}\right], \label{gradf} \eeq
where $\mathbb{1} \in \R^{n-1}$ is the column vector of all ones.  
Via \eqref{sgn} and \eqref{gradf}, \eqref{yk} is simply
\beq\label{yk1}
y_{k-1}= \left[\begin{array}{c} (-1)^k 2a \\ \mathbb{0} \end{array}\right],
\eeq
where $ \mathbb{0} \in \R^{n-1}$ is the column vector of all zeros.
Note that from \eqref{update} and \eqref{sk} it is immediate that for any $k\ge 1$
\beq\label{std}
s_{k-1}=t_{k-1}d_{k-1}.
\eeq
For $i=2,\hdots,n$, let 
$$  \theta^{(i)}_{k-1}  = \arctan\left (\dfrac{d_{k-1}^{(i)}}{d_{k-1}^{(1)}}\right) ,
$$
with $\theta_{k-1}^{(i)} \in[-\pi/2,\pi/2]$. Note that	$|\theta_{k-1}^{(i)}|$ is 
the acute angle between $d_{k-1}$ and the $x^{(1)}$ axis when it is projected onto the $(x^{(1)}, x^{(i)})$ plane.
	 From \eqref{d0} and \eqref{gradf} we have
	\beq 
	 \frac{1}{a} = \tan \theta_0^{(2)} = \tan \theta_0^{(3)} = \hdots = \tan \theta_0^{(n)}. \label{b0}
	\eeq 
	 The assumption of the initial inverse Hessian approximation being a multiple of the identity is embedded in the definition \eqref{hknw}, and therefore we know that $d_{k-1}$ (and consequently $s_{k-1}$) is in the subspace spanned by the two gradients in \eqref{gradf} (see \cite[Lemma 2.1]{rh03}). Since both gradients are symmetric w.r.t. the components $x^{(2)},\hdots, x^{(n)}$, it follows that $d_{k-1}$ has the same property. The same symmetry holds in the definition of the objective function \eqref{fdef}. Since \eqref{b0} holds, we conclude inductively that, for $k>1$,
	$
	\tan \theta_{k-1}^{(2)} = \tan \theta_{k-1}^{(3)} = \hdots = \tan \theta_{k-1}^{(n)}. 
	$
	So, let us simply write
	\beq\label{bi}
	b_{k-1} = \tan \theta_{k-1} = \dfrac{d_{k-1}^{(i)}}{d_{k-1}^{(1)}} = \dfrac{s_{k-1}^{(i)}}{s_{k-1}^{(1)}}, \mathrm{~~for ~all~ } i=2,\hdots,n.
	\eeq	 
From \eqref{yk1} we have 
	\beq \label{ys} 
	s^T_{k-1}y_{k-1} = (-1)^k 2as_{k-1}^{(1)},
	\eeq 
	so we can rewrite $V_{k-1}$ in \eqref{Vk}  in terms of $b_{k-1}$ as 
	\beq\label{Vkb}
	V_{k-1} =
	\left[
	\begin{array}{c|c}
		0 & -b_{k-1}\mathbb{1}^T\\
		\hline
		\mathbb{0} & I_{{n-1}}
	\end{array}
	\right].
	\eeq
This leads us to write $H_{k}$ in \eqref{hknw} as
	$$H_{k} =\dfrac{s_{k-1}^Ty_{k-1}}{y^T_{k-1}y_{k-1}} \left[
	\begin{array}{c|c}
	0 & \mathbb{0}^T\\
	\hline
	\mathbb{0} & b_{k-1}^2 \mathbb{1}\mathbb{1}^T+I_{{n-1}   }
	\end{array}
	\right]+  
	\dfrac{(s_{k-1}^{(1)})^2}{s^T_{k-1}y_{k-1}}\left[
	\begin{array}{c|c}
	1 & b_{k-1}\mathbb{1}^T \\
	\hline
	b_{k-1}\mathbb{1}& b_{k-1}^2 \mathbb{1}\mathbb{1}^T
	\end{array}
	\right].
	$$
	From \eqref{ys} we can see that the fractions in front of the first and second matrices are the same, i.e.,
	\beq \label{gam_abs}
	\dfrac{s_{k-1}^Ty_{k-1}}{y^T_{k-1}y_{k-1}}= \dfrac{(s_{k-1}^{(1)})^2}{s^T_{k-1}y_{k-1}} =\dfrac{|s^{(1)}_{k-1}|}{2a}.
	\eeq
	Hence, we obtain the following much more compact form
	\beq\label{hkgen}
	H_{k} =\gamma_k  \left[
	\begin{array}{c|c}
		1 & b_{k-1}\mathbb{1}^T \\
		\hline
		b_{k-1}\mathbb{1}& 2b_{k-1}^2 \mathbb{1}\mathbb{1}^T+I_{n-1} 
	\end{array}
	\right] ,
	\eeq
	where 
	\beq \label{gammadef}
	\gamma_k =  \dfrac{|s^{(1)}_{k-1}|}{2a} 
	\eeq
	is the scale factor in \eqref{Hk0def}.
	Finally, with the gradient defined in \eqref{gradf} we can compute the direction generated by Algorithm~1 in \eqref{dir} as
\beq\label{dkn}
d_k= -\dfrac{|s^{(1)}_{k-1}|}{2a}
\left[\begin{array}{c}  
	(-1)^k a + (n-1)b_{k-1} \\ \Big((-1)^k ab_{k-1} + 2(n-1)b_{k-1}^2+1\Big)\mathbb{1}
\end{array}\right]. 
\eeq

So, from definition \eqref{bi} we can write $b_k$ recursively as 
\beq\label{bk}
b_k = \dfrac{(-1)^k ab_{k-1} + 2(n-1)b_{k-1}^2+1}{(-1)^k a + (n-1)b_{k-1}}.
\eeq   
\subsection{Existence of Armijo-Wolfe Steps when 
		$\sqrt{3(n-1)}\leq a$}\label{subsec:AWexist}
In the next lemma we prove that if $ \sqrt{3(n-1)} \le a$, then the $\{b_k\}$ alternate in sign
with $|b_k|\leq 1/a$.
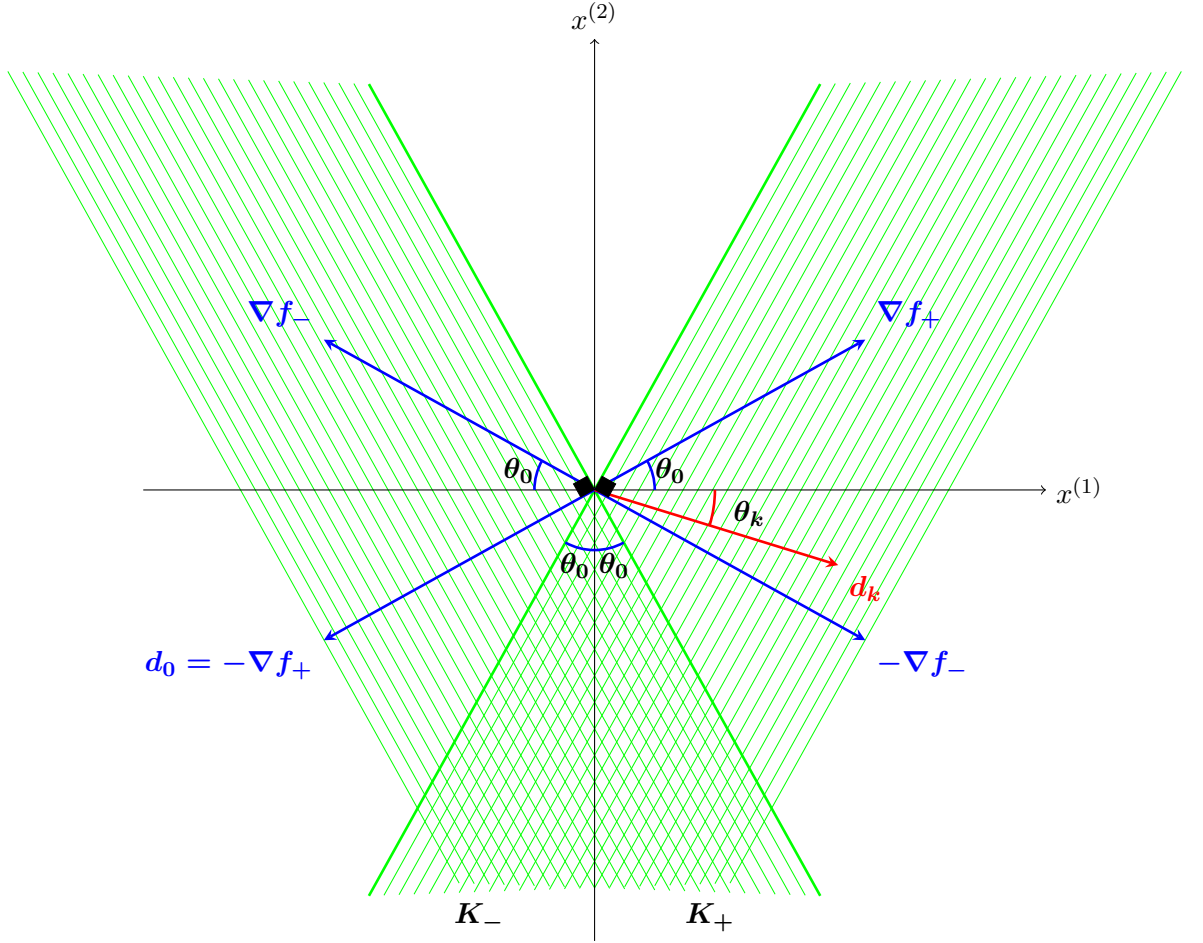
\begin{figure}
	\centering
\begin{tikzpicture}[scale=2]
\usetikzlibrary{patterns}
\newcommand{\x}{3cm}
\newcommand{\g}{1.8cm}
\newcommand{\sq}{0.05cm}
\newcommand{\sqq}{0.09cm}
\coordinate (origin) at (0, 0);
\coordinate [label={above right:$\boldsymbol{K_{-}}$}] (K_) at (-1, -3);
\coordinate [label={above left:$\boldsymbol{K_{+}}$}] (K+) at (+1, -3);
\coordinate (x) at (\x,0);
\coordinate (mx) at(-\x,0);
\coordinate (y) at (0,\x);
\coordinate (my) at (0,-\x);
\coordinate (fp) at (\g,1);
\coordinate (mfp) at (-\g,-1);
\coordinate (mfm) at (\g,-1);
\coordinate (mf) at (-\g,1);

\def\n{24};

\def\xo{1.5}; 
\def\xoo{-1.5};
\def\yo{-1.5*\g};
\def\yoo{1.5*\g};
\def\xq{-1.5}; 
\def\xqo{1.5};
\def\yq{-1.5*\g};
\def\yqo{1.5*\g};

\coordinate (pmr) at (\xo,\yo);
\coordinate (mpr) at (\xoo,\yoo);
\coordinate (mmr)  at (\xq,\yq);
\coordinate (ppr) at (\xqo,\yqo);

\draw[line width=1pt,green](pmr)--(mpr) ; 
\foreach \x in {1,..., \n}{
	\def\xn{{ -.1*\x -\xo}};  
	\def\yn{ { .1*\x - \yo} }; 
	
	\def\xnn{{ -.1*\x -\xoo}};  
	\def\ynn{ { .1*\x - \yoo } }; 
	
	\draw [green,line width=0mm] ({\xn},{\yn}) -- ({\xnn},{\ynn});
}

\draw[line width=1pt,green](mmr)--(ppr) ; 
\foreach \x in {1,..., \n}{
	\def\xn{{ .1*\x -\xq}};  
	\def\yn{ { .1*\x - \yq} }; 
	
	\def\xnn{{ .1*\x -\xqo}};  
	\def\ynn{ { .1*\x - \yqo } }; 
	
	\draw [green,line width=0mm] ({\xn},{\yn}) -- ({\xnn},{\ynn});
}

\coordinate (d) at (.9*\g,-0.5);
\coordinate (xk) at (0, .05*\x);

\draw[->] (mx)-- (x) node[right]{$x^{(1)}$};
\draw[->] (my)--(y) node[above]{$x^{(2)}$};

\draw[line width=1pt,blue,-stealth](origin)--(fp) node[anchor=south west]{$\boldsymbol{\nabla f_+ }$};
\draw[line width=1pt,blue,-stealth](origin)--(mfp) node[anchor=north east]{$\boldsymbol{d_0=-\nabla f_+}$};
\draw[line width=1pt,blue,-stealth](origin)--(mf) node[anchor=south east]{$\boldsymbol{\nabla f_- }$};
\draw[line width=1pt,blue,-stealth](origin)--(mfm) node[anchor=north west]{$\boldsymbol{-\nabla f_-}$};
\draw[line width=1pt,red,-stealth](origin)--(d) node[anchor=north west]{$\boldsymbol{d_{k}}$};

\draw[fill] (0,0) -- ++ (\sq, \sqq) -- ++ (\sqq, -\sq) -- ++ (-\sq, -\sqq);

\draw[fill] (0,0) -- ++ (-\sq,\sqq) -- ++ (-\sqq,-\sq) -- ++ (\sq,-\sqq) ;

\draw[line width=1pt] pic["$\boldsymbol{\theta_0}$", draw=blue, -, angle eccentricity=1.3, angle radius=.8cm] {angle =x--origin--fp};

\draw[line width=1pt] pic["$\boldsymbol{\theta_0}$", draw=blue, -, angle eccentricity=1.3, angle radius=.8cm] {angle =mf--origin--mx};

\draw[line width=1pt] pic["$\boldsymbol{\theta_0}$", draw=blue, -, angle eccentricity=1.3, angle radius=.8cm] {angle =mmr--origin--my};

\draw[line width=1pt] pic["$\boldsymbol{\theta_0}$", draw=blue, -, angle eccentricity=1.3, angle radius=.8cm] {angle =my--origin--pmr};

\draw[line width=1pt] pic["$\boldsymbol{\theta_{k}}$", draw=red, -, angle eccentricity=1.3, angle radius=1.6cm, ] {angle =d--origin--x};

\end{tikzpicture}
\caption{{\bf Angles of Search Directions.} Let $n=2$, 
		let  $\grad f_+ = [a~1]^T$ and let $\grad f_- = [-a~1]^T$, so, since $x_0^{(1)} > 0$ by assumption, we have $d_0=-\grad f_+$.
	    It follows from Lemma \ref{lemma1} that $b_k = d_{k}^{(2)}/d_{k}^{(1)}$
	    alternates in sign for $k=1,2,\hdots$, with absolute value bounded above by $1/a$,
	    and hence that $\theta_{k} = \arctan(b_k)$ alternates in sign for $k=1,2,\hdots$, with
	    $|\theta_k|$, the acute angle between the $x^{(1)}$ axis and the search direction $d_{k}$, bounded above by $\theta_0$. Furthermore,  
		Lemma \ref{genral_dir} states that the function $f$ is unbounded below along all directions
		in the open cones $K_-$ and $K_+$, and bounded below along all other directions (except the vertical axis).  Note, however, that points satisfying the Wolfe condition exist along directions $d\in K_+$ emanating  from iterates on the left side of the $x^{(2)}$ axis, but not along directions $d\in K_-$ emanating from the left side, because the former cross the $x^{(2)}$ axis and the latter do not, and vice versa. 
	Finally, Theorem \ref{sp_dir} implies that, under the assumption $a\geq \sqrt{3}$, we have $|\theta_{k}| \le \theta_0 \le \pi/6 $, for all $k > 0$ (see the discussion after
		the theorem), so $d_{k}$ does not lie in $K_-$ or in $K_+$ and hence the algorithm does not terminate.
}\label{fig:plane}\end{figure}
\begin{lemma} \label{lemma1}
	Suppose $ \sqrt{3(n-1)} \le a$ . Define $b_k$ as in  \eqref{bk}  with $b_0=1/a$. Then $|b_k| \le 1/a$ and furthermore $\{b_k\}$ 
	alternates in sign with
	\beq\label{abk}
	|b_k| = \dfrac{1+(n-1)b_{k-1}^2}{a - (n-1)|b_{k-1}|} -|b_{k-1}|.
	\eeq
\end{lemma} 
\begin{proof}
	See Appendix \ref{appendA} for the proof.
\end{proof}
Putting \eqref{bk} and \eqref{abk} together we can rewrite \eqref{dkn} as 
\beq
\label{dksk_1}
d_k= - \dfrac{|s^{(1)}_{k-1}|}{2a}(a-(n-1)|b_{k-1}|)
\left[\begin{array}{c} 
	(-1)^k \\ |b_k|\mathbb{1}   
\end{array}\right].
\eeq
Before stating the main result of this section we give the following simple lemma.
\begin{lemma} \label{genral_dir}
Let $x\in\R^n$ be given, define
	\beq
	\label{beta_dir}
	 d_+= -\left[\begin{array}{c} 
		1 \\  \beta\mathbb{1}   
	\end{array}\right]  \mathrm{~~and~~}	d_-= -\left[\begin{array}{c} 
	-1 \\ \beta \mathbb{1}   
\end{array}\right] ,
	\eeq
	where $\beta > 0$, and define $f$ by \eqref{fdef}.  Let $d$ be either $d_+$ or $d_-$. Then
	$h(t) = f(x + t d)  -f(x) $ is unbounded below if and only if  $\dfrac{a}{n-1}< \beta$. 
\end{lemma}
\begin{proof}
    We have
	$$h(t)=a|x^{(1)} \pm t| -a|x^{(1)}|-(n-1)\beta t .$$
   So,
	$$ \left(a-(n-1)\beta\right) t -2a|x^{(1)}| < h(t)  < \left(a-(n-1)\beta\right) t.$$
	The result follows.
\end{proof}
Note that stating that $h$ is unbounded below is not equivalent to saying that Armijo-Wolfe points do not exist along the direction $d$ emanating from $x$. Such points exist if and only
if the sign of $d^{(1)}$ is opposite to the sign of $x^{(1)}$.
\begin{theorem}\label{sp_dir} 
When Algorithm 1 is applied to \eqref{fdef} with $ \sqrt{3(n-1)} \le a$, using any Armijo-Wolfe line search,   with any starting point $x_0$ such that $x_0^{(1)} \neq 0$,
	the method generates directions $d_k$ that are nonnegative scalar multiples of $d_+$ or $d_-$, defined in \eqref{beta_dir},
	with $\beta < a/(n-1)$. It follows that the steplength $t_k$ satisfying the Armijo and Wolfe conditions \eqref{armijo_cond}
	and \eqref{wolfe_cond} always exist and hence the method never terminates.
\end{theorem}
\begin{proof}
The proof is by induction on $k$. Without loss of generality assume $x_0^{(1)} > 0$, so $d_0=-\grad f(x_0)=a d_+$ with
	$\beta=1/a$. Since $ \sqrt{3(n-1)} \le a$, we have $1/a <a/(n-1) $ and hence the initial Armijo-Wolfe steplength $t_0$ exists by Lemma \ref{genral_dir}.
	Now, suppose that the result holds for all $j<k$, so
	  $d_k$ in \eqref{dksk_1} is well defined. Since by Lemma~\ref{lemma1} we know that $|b_{k-1}| \le 1/a \le a/(n-1) $,
	  the leading scalar in \eqref{dksk_1} is negative and therefore $d_k$ is a nonnegative scalar multiple 
	  of $d_+$ or $ d_-$ with $\beta = |b_k| \leq 1/a < a/(n-1)$. 
	  Hence $f$ is bounded below along the direction $d_k$ emanating from $x_k$ and so there exists $t_k$ satisfying the Armijo and Wolfe conditions at iteration $k$, which implies that the algorithm does not terminate at iteration $k$.
\end{proof}
Using Figure \ref{fig:plane} we can provide an alternative informal geometrical proof for Theorem \ref{sp_dir}.  We have 
$$ \dfrac{1}{a} \le \dfrac{1}{\sqrt{3}}  ~~\RA~~ \theta_0 = \arctan \dfrac{1}{a} ~\le ~ \arctan \dfrac{1}{\sqrt{3}} = \dfrac{\pi}{6}.$$	
According to Lemma \ref{lemma1}, we have $ |b_{k}|\le 1/a$, and so, $|\theta_{k}| \leq \theta_0$  and hence, 
$$2\theta_0 + |\theta_{k}| \le \dfrac{\pi}{2}.$$
It follows (see Figure \ref{fig:plane}) that $d_{k} \notin K_+ \cup K_-$.   
This means that the method never generates a direction along which $f$ is unbounded below. 

However, Theorem \ref{sp_dir} does not imply that Algorithm 1 converges to a non-optimal point under the
assumption that $\sqrt{3(n-1)} \le a$, because the existence of Armijo-Wolfe steps $t_k$ for all $k$ does not
imply that the sequence $\{f(x_k)\}$ is bounded below. This issue is addressed in the next section.
\section{Failure of Scaled Memoryless BFGS}\label{sec:theory}
\subsection{Convergence of the Absolute Value of the Normalized Search Direction when $ 2\sqrt{n-1} \le a$}\label{subsec:maincondition}
Define 
\beq \label{b}
b= \dfrac{a-\sqrt{a^2-3(n-1)}}{3(n-1)}
\eeq
and note that when $ \sqrt{3(n-1)} \le a $, then
\[       \dfrac{1}{2a} \le b \le \dfrac{1}{a} \label{b-ineqs}. \]
Next we show the sequence $\{|b_k|\}$ converges to $b$
under a slightly stronger assumption. 
\begin{theorem}\label{conv}
	For $ 2\sqrt{n-1} \le a$ the sequence defined by \eqref{abk} converges and moreover
	$$\lim_{k \to \infty} |b_k| = b. $$ 
\end{theorem}
\begin{proof}
	See Appendix \ref{appendB} for the proof.
\end{proof}

Note that the convergence result established in this theorem does not require any assumption of symmetry with respect to variables
$2,3,\ldots,n$ in the initial point $x_0$. 
The only assumption on $x_0$ is that $x_0^{(1)}>0$. We need $x_0^{(1)}\not = 0$ so that $f$ is differentiable at $x_0$; the assumption
on the sign is purely for convenience.

\textbf{Assumption 1.} For the subsequent theoretical analysis we assume that $$2\sqrt{n-1}  \le a.$$
	With this assumption, as a direct implication of Theorem \ref{conv}, for any given positive $\epsilon$ there exists $K$ such that for $k \ge K$
	we have
	\beq\label{n_1b_eps}
	||b_k| - b| < \dfrac{\epsilon}{n-1}.
	\eeq
As we showed in Lemma \ref{lemma1}, for  $k \ge 0$ we have $|b_k| \le 1/a$ and therefore 
\beq\label{pos}
\dfrac{3(n-1)}{a}\le a-\dfrac{n-1}{a} \le a -(n-1)|b_k|.
\eeq
Thus, $a -(n-1)|b_k|$ is positive and bounded away from zero.

Since $|b_k|$ converges by Theorem~\ref{conv}, we see that in the limit the normalized direction $d_k/\|d_k\|_2$
alternates between two limiting directions.
For an illustration, see Figures \ref{fig:gl1}
and \ref{fig:gl2}. It is this property that allows
us to establish, under some subsequent assumptions, that scaled memoryless BFGS generates iterates $x_k$ for which $f(x_k)$ is bounded below even though $f$ is unbounded below.

\subsection{Dependence on the Armijo Condition}\label{subsec:armijo}
Combining \eqref{gradf} and \eqref{dksk_1} we get
\beq \label{dir_der}
 \grad f(x_k)^T d_k = -|d^{(1)}_k|
 \left[\begin{array}{c}
 (-1)^k a \\ \mathbb{1} 
\end{array}\right]^T 
 \left[\begin{array}{c}
     (-1)^k \\ |b_k|\mathbb{1}    
 \end{array}\right] = -|d^{(1)}_k|\left(a+\left(n-1\right)|b_k|\right),
\eeq
so the Armijo condition \eqref{armijo_cond} with $t=t_k$ at iteration $k$ is 
\beq \label{real_armijo}
c_1t_k|d^{(1)}_k|\left(a+\left(n-1\right)|b_k|\right)\le f(x_k) - f(x_k+t_kd_k).
\eeq
If $t_k$ satisfies the Wolfe condition, i.e. $t_k$ is large enough that the sign change \eqref{sgn} occurs, then we must have
\beq\label{x1td}
|x^{(1)}_k| < t_k|d^{(1)}_k|.
\eeq
Given this we can derive $f(x_k) - f(x_k+t_kd_k) $  using
the definition of $b_k$ in \eqref{bi} as follows:
\begin{align*} 
&f(x_k) - f(x_k+t_kd_k) 
=2a|x^{(1)}_k|-\left(a-(n-1)|b_k|\right)t_k|d^{(1)}_k| \numberthis \label{num}.
\end{align*}
By defining $\varphi_k$ as follows
\beq\label{C}
\varphi_k = \dfrac{c_1\left(a+(n-1)|b_k|\right) +a-(n-1)|b_k| }{2a},
\eeq
we can restate  the Armijo condition in the following lemma.
\begin{lemma} \label{arm_cond}
	Suppose $t_k$ satisfies the Wolfe condition \eqref{sgn}. Then for $t_k$ to satisfy the Armijo condition \eqref{real_armijo} we must have
	\beq\label{arm_e_s}
	\varphi_kt_k|d^{(1)}_k| \le |x^{(1)}_k|.
	\eeq
\end{lemma}
\begin{proof}
	Combining \eqref{num} and \eqref{real_armijo} we get
	\[
	c_1t_k|d^{(1)}_k|\left(a+\left(n-1\right)|b_k|\right) \le 2a|x^{(1)}_k|-\left(a-\left(n-1\right)|b_k|\right)t_k|d^{(1)}_k|,
	\]
	and using the definition of $\varphi_k$ in \eqref{C}, \eqref{arm_e_s} follows.
\end{proof}
From \eqref{x1td}  and  \eqref{arm_e_s} we see that $\varphi_k$ is the ratio of the lower bound and the upper bound
on the steplength $t_k$ provided by the Wolfe and Armijo conditions respectively.
The next lemma provides bounds 
on $\varphi_k$.

\begin{lemma}
	\beq\label{nice}
	\dfrac{(n-1)|b_k|}{a} < \varphi_k .
	\eeq
\end{lemma}
\begin{proof}
	Using Lemma \ref{lemma1} we know $3(n-1)|b_k| \le a$ for all $k$, and so 
	$$2(n-1)|b_k| \le a - (n-1)|b_k|,$$
	and since 
	$$ \dfrac{a - (n-1)|b_k|}{2a} =  \varphi_k -c_1\dfrac{a+(n-1)|b_k|}{2a},$$ 
	and $c_1 > 0$, \eqref{nice} follows.
\end{proof}

\begin{corollary}  \label{s_varphi}
	For $k \ge 1$ we have
	\beq\label{armsk_1}
	|s^{(1)}_k|\le |s^{(1)}_{k-1}|\dfrac{1-\varphi_{k-1}}{\varphi_k}.
	\eeq
\end{corollary}
\begin{proof}
	Summing the Armijo inequality  \eqref{arm_e_s} for two consecutive iterations we obtain 
	$$
	|s^{(1)}_{k-1}|\varphi_{k-1} + |s^{(1)}_k|\varphi_k \le |x^{(1)}_{k-1}| + |x^{(1)}_k|,
	$$ 
	and noticing that the R.H.S., according to \eqref{sgnxk_1}, is equal to $|s^{(1)}_{k-1}|$ we get \eqref{armsk_1}. 
\end{proof}
\begin{lemma} \label{Kdef}
	For any given $\epsilon$ $>0$ let $K$ be the smallest integer such that for any $k \ge K$, \eqref{n_1b_eps} holds. Then for all $N > K$ we have  
	\beq\label{me}
	f(x_K) - f(x_N) < a|x^{(1)}_K|+ \left(
	\left(n-1\right)b+\epsilon \right)\sum_{k=K}^{N-1}|s^{(1)}_k|.  
	\eeq
\end{lemma}
\begin{proof}
	Using $t_kd_k= s_k$ and $x_{k+1} = x_k+s_k$ in \eqref{num} and then applying \eqref{n_1b_eps} we obtain
	\beq \label{approx}
	f(x_k) - f(x_{k+1}) < 2a|x^{(1)}_k|-a|s^{(1)}_k| +\left(\left(n-1\right)b+\epsilon\right)|s^{(1)}_k| . 
	\eeq
	Summing up \eqref{approx} from $k=K$ to $k=N-1$ and recalling \eqref{sgnxk_1},
	we get
	\begin{align*}
	&f(x_K) - f(x_N) <\\
	&a\sum_{k=K}^{N-1}|s^{(1)}_k| + a|x^{(1)}_K|- a|x^{(1)}_N|  -a\sum_{k=K}^{N-1}|s^{(1)}_k| +\left(
	\left(n-1\right)b+\epsilon \right)\sum_{k=K}^{N-1}|s^{(1)}_k|.
	\end{align*}
	Canceling the first and fourth terms and dropping $-a|x_N|$, we arrive at \eqref{me}.
\end{proof}

From applying Theorem \ref{conv} to the definition of $\varphi_k$ in \eqref{C} it is immediate that $\{\varphi_k\}$ converges. Let
\beq
\label{phi-def}
\varphi= \dfrac{c_1\left( a+(n-1)b\right) + a-(n-1)b}{2a},
\eeq
so
\beq\label{limlim}
\lim_{k \to \infty} \varphi_k = \varphi.
\eeq

\begin{lemma} \label{eps_lem}
	Assume
	\beq\label{epsilon} 
	0 < \epsilon \le \dfrac{\sqrt{a^2-3(n-1)}}{3},
	\eeq 
	and let $K$ be defined as in Lemma \ref{Kdef}. Then for any $k\ge K$ we have
	\beq\label{varfi_err}
	\left|\dfrac{1-\varphi_{k-1}}{\varphi_k} - \dfrac{1-\varphi}{\varphi}\right| < \dfrac{15}{a}\epsilon.
	\eeq
\end{lemma}
\begin{proof}
	By rearranging terms in \eqref{b} and using \eqref{epsilon} we get 
	\beq \label{newineq}
	(n-1)b+\epsilon \leq  (n-1)b + \dfrac{\sqrt{a^2-3(n-1)}}{3} = \dfrac{a}{3}. 
	\eeq
	Using \eqref{n_1b_eps} and \eqref{newineq},
	for $k \ge K$ we have
	\[
	0 < a-(n-1)b-\epsilon < a-(n-1)|b_k|.
	\]
	Combining this with  \eqref{nice}  
	we get
	$$0 < \dfrac{a-(n-1)b-\epsilon}{2a} <  \varphi_k < 1.$$	
	Hence, 
	$$1 < \dfrac{1}{\varphi_k } < \dfrac{2a}{a-(n-1)b-\epsilon} \leq \dfrac{2a}{a-\frac{a}{3}} = 3.$$
	Since $0<c_1<1$, from \eqref{n_1b_eps}, \eqref{C}, \eqref{phi-def} and \eqref{limlim} we get
	\[|\varphi_k - \varphi| < \dfrac{(1+c_1)\epsilon}{2a} < \dfrac{\epsilon}{a}. \]
	So,
	\begin{align*}
	&\left|\dfrac{1-\varphi_{k-1}}{\varphi_k} - \dfrac{1-\varphi}{\varphi} \right| 
	= \left|\dfrac{1}{\varphi_k} -1 + \dfrac{\varphi_k -\varphi_{k-1}}{\varphi_k} - \dfrac{1}{\varphi} +1 \right|  \\
	&< \left|\dfrac{\varphi -\varphi_{k}}{\varphi_k\varphi} \right| + \left|\dfrac{\varphi_k -\varphi_{k-1}}{\varphi_k} \right| 
	< \dfrac{\epsilon}{a\varphi_k} \left (\dfrac{1}{\varphi}+2 \right).
	\end{align*}
	Note that $ 1 < 1/\varphi_k  <  3$ applies to all $\varphi_k$ (as well as the limit $\varphi$) with $k \ge K$, and therefore we conclude \eqref{varfi_err}.
\end{proof}

Let 
\beq\label{var_eps}
\psi_{\epsilon}=\dfrac{1-\varphi}{\varphi}+\dfrac{15}{a}\epsilon.
\eeq
If Lemma \ref{eps_lem} applies then
from \eqref{armsk_1} and \eqref{varfi_err} we conclude
\beq\label{armsk_phi}
|s^{(1)}_k|< \psi_{\epsilon} |s^{(1)}_{k-1}|.
\eeq
That is to say, with $\epsilon$ satisfying \eqref{epsilon}, after at most $K$ iterations, \eqref{armsk_phi} holds. Consequently, with the additional assumption $\psi_{\epsilon} < 1$, we obtain  
\beq\label{sk_psi}
\sum_{k=K}^{N-1}|s^{(1)}_k| < |s^{(1)}_K|\dfrac{1}{1-\psi_{\epsilon}}.
\eeq
Now we can prove the main result of this subsection. Recall that $c_1 < 1$.
\begin{theorem}\label{c_1thm}
	Suppose $c_1$ is chosen large enough that
	\beq
	\dfrac{1}{c_1} -1< \dfrac{a}{(n-1)b} \label{c1failcond}
	\eeq
	holds. Then, using any Armijo-Wolfe line search with any starting point $x_0$ with $x_0^{(1)} \neq 0$,
	scaled memoryless BFGS applied to \eqref{fdef} fails in the sense that $f(x_N)$ is bounded below as $N \to \infty$.
\end{theorem}
\begin{proof}
	It follows from \eqref{c1failcond} and \eqref{phi-def} that $\varphi > 1/2$.
	Therefore, using \eqref{var_eps}, we can  choose $\epsilon$ small enough such that $\psi_{\epsilon} < 1$ 
	holds in addition to \eqref{epsilon}. Applying Lemmas \ref{Kdef} and \ref{eps_lem}, we conclude
	that there exists $K$ such that for for any $N > K$,  \eqref{sk_psi} holds, and,
	substituting this into \eqref{me} we get
	\beq\label{tot_dec_1}
	f(x_K) - f(x_N) <  a|x^{(1)}_K| + |s^{(1)}_K|\dfrac{
		\left(n-1\right)b+\epsilon}{1-\psi_{\epsilon}}.
	\eeq
	This establishes that $f(x_N)$ is bounded below  for all $N > K$.
\end{proof}
Using \eqref{b} we see that the failure condition \eqref{c1failcond} for scaled memoryless BFGS with any Armijo-Wolfe line search
applied to \eqref{fdef} is equivalent to
\beq\label{lbf_fail}
\dfrac{1-c_1}{c_1}(n-1) ~<~ a^2+a\sqrt{a^2 - 3(n-1)}.
\eeq
The corresponding failure condition for the gradient method on the same function, again using any
Armijo-Wolfe line search, is, as we showed in \cite{AO18},
\beq\label{gm_fail}
\dfrac{1-c_1}{c_1}(n-1) ~ < ~a^2.
\eeq
Hence, scaled memoryless BFGS fails under a  \emph{weaker} condition relating $a$ to the Armijo parameter than the condition for failure of the gradient
method on the same function with the same line search conditions.
Indeed, Assumption 1 implies 
\[
a^2+a\sqrt{a^2 - 3(n-1)} \geq 4(n-1) + 2\sqrt{n-1}\sqrt{n-1} = 6 (n-1).
\]
So, if the Armijo parameter $c_1 \geq 1/7$, then \eqref{lbf_fail} holds.  
In contrast, the same assumption implies that if $c_1\geq 1/5$, then
\eqref{gm_fail} holds. So, scaled memoryless BFGS with any Armijo-Wolfe line search applied to \eqref{fdef} fails 
under a weaker condition on the Armijo parameter than the gradient method does.

\subsection{Results for a specific Armijo-Wolfe line search, independent of the Armijo parameter}\label{subsec:specificLS}
Considering only the first component of the direction $d_k$ in \eqref{dksk_1} we have
\beq\label{sksk_1}
\dfrac{2a}{a-(n-1)|b_{k-1}|} |d^{(1)}_k| =|s^{(1)}_{k-1}|.
\eeq
Using \eqref{std}, it follows that if
\beq\label{tkbd}
t_k < \dfrac{2a}{a-(n-1)|b_{k-1}|},  
\eeq
we have $|s^{(1)}_k| <|s^{(1)}_{k-1}|$.
Note that the R.H.S. of \eqref{tkbd} is greater than two. However,
as shown in the next lemma,
except at the initial iteration ($k=0$), $t = 2$ 
is always large enough to satisfy the Wolfe condition, implying that there exists $t\leq 2$ satisfying
both the Armijo and Wolfe conditions. 

\begin{lemma}\label{thm_aw} 
For $k\geq 1$, the steplength
$t_k = 2$ always satisfies the Wolfe condition \eqref{sgn}, i.e., we have
\beq\label{sgnx1}
|x^{(1)}_k| < 2|d^{(1)}_k|.
\eeq
\end{lemma}
\begin{proof}
Since $k\geq 1$, we know that the Armijo and Wolfe conditions hold at iteration $k-1$ by definition
of Algorithm 1. So, using  \eqref{arm_e_s} and \eqref{std} we have
\beq\label{arm_k_1}
\varphi_{k-1}|s^{(1)}_{k-1}| \le |x^{(1)}_{k-1}|.
\eeq
Using the inequality \eqref{nice} in the L.H.S.\   and the equality \eqref{sgnxk_1} in the R.H.S.\  we get
$$ \dfrac{(n-1)|b_{k-1}|}{a} |s^{(1)}_{k-1}| <  |s^{(1)}_{k-1}| - |x^{(1)}_k|,$$
i.e. 
\[ |x^{(1)}_k| < |s^{(1)}_{k-1}|\dfrac{a - (n-1)|b_{k-1}|}{a}.\]
Substituting \eqref{sksk_1} into the R.H.S., we obtain \eqref{sgnx1}.
\end{proof} 

Now let us focus on the Armijo-Wolfe bracketing line search given in \cite{LO13,AO18}, which we state
here for convenience.

\begin{algorithm} [H]
 \textbf{Algorithm 2 (Armijo-Wolfe Bracketing Line Search)}
\begin{algorithmic}
\State $\alpha \leftarrow 0$
\State $\beta \leftarrow +\infty $
\State $t\leftarrow 1$
\While  {true} 
      \If {the Armijo condition \eqref{armijo_cond}) fails}
            \State $\beta \leftarrow t $
      \ElsIf  {the Wolfe condition \eqref{wolfe_cond} fails}
           \State $\alpha \leftarrow t $
      \Else
           \State  stop and return $t$
     \EndIf
     
     \If {$ \beta < +\infty $}
            \State $t \leftarrow (\alpha+\beta)/2$
      \Else
           \State $t \leftarrow 2\alpha$
     \EndIf
\EndWhile
\end{algorithmic}
\end{algorithm}
It is known from the results in \cite{LO13} that provided $f$ is bounded below along $d_{k-1}$ (as we already
established must hold for directions generated by Algorithm 1), 
the Armijo-Wolfe bracketing line search will terminate with a steplength $t$ satisfying both conditions.
In the following lemma we show that if we use this line search, it always generates  $t_k\leq 2$ for $k\geq 1$.

\begin{lemma}\label{brakLS}
When scaled memoryless BFGS is applied to \eqref{fdef}, using Algorithm 2 
it always returns steplength $t_k \le 2$ for $k \ge 1$.
\begin{proof}
	The line search begins with the unit step. If this step, $t=1$, does not satisfy the Armijo condition \eqref{armijo_cond},
	then the step is contracted, so the final step is less than one. On the other hand, if $t=1$ satisfies \eqref{armijo_cond},
	then the line search checks whether the Wolfe condition \eqref{wolfe_cond} is satisfied too. 
	If it is, then the line search quits; if not, the step is doubled
	and hence the line search next checks whether $t=2$ satisfies \eqref{wolfe_cond}. At the initial iteration $(k=0)$,
	several doublings might be needed before \eqref{wolfe_cond} is eventually satisfied. But for subsequent steps
	($k\geq 1$), we know that $t=2$ must satisfy the Wolfe condition, so the final step must satisfy $t_k=2$ (if
	$t=2$ satisfies \eqref{armijo_cond}) or $t_k < 2$ (otherwise). Thus, for $k\geq 1$ we always have $t_k\leq 2$.
\end{proof}
\end{lemma}
Now we can present the main result of this subsection: using a line search with the property just described, the optimization method fails.
\begin{theorem} \label{del_thm}
If  scaled memoryless BFGS is applied to \eqref{fdef}, using
 	an Armijo-Wolfe line search that satisfies $t_k\leq 2$ for $k\geq 1$, such as
	Algorithm 2 then the method fails in the sense that
$f(x_N)$ is bounded below as $N \to \infty$.
\end{theorem}
\begin{proof}	
Recalling $t_{k+1}d^{(1)}_{k+1}=s^{(1)}_{k+1}$ again, using \eqref{sksk_1} and $t_{k+1} \le 2$ we find that
\beq\label{t22}
 |s^{(1)}_{k+1}| \le  \dfrac{a-(n-1)|b_k|}{a}|s^{(1)}_k|.
\eeq
Let $\epsilon>0$ satisfy
\[
     \delta_{\epsilon} \equiv \dfrac{a-(n-1)b}{a} + \dfrac{\epsilon}{a} < 1.
\] 
Define $K$ as in Lemma \ref{Kdef}, so that \eqref{n_1b_eps} holds, and hence
$$\dfrac{a-(n-1)|b_{k}|}{a} < \delta_{\epsilon}.$$
Applying this inequality to \eqref{t22} we get
\beq\label{t2}
|s^{(1)}_{k+1}| \le \delta_{\epsilon}|s^{(1)}_{k}|,
\eeq
and since $\delta_{\epsilon} < 1$ we have
\beq\label{sum_conv}
\sum_{k=K}^{N-1}|s^{(1)}_k| < |s^{(1)}_K|\dfrac{1}{1-\delta_{\epsilon}}.
\eeq
By substituting this into \eqref{me} we get
\[
f(x_K) - f(x_N) <  a|x^{(1)}_K| + |s^{(1)}_K|\dfrac{
	\left(n-1\right)b+\epsilon}{1-\delta_{\epsilon}},
\]
which shows $f(x_N)$ is bounded below.
\end{proof}
Finally, we have the following corollary to Theorems \ref{c_1thm} and \ref{del_thm}. 
Recall that $\gamma_k$ is the scale factor (see \eqref{gammadef}).
\begin{corollary}
\label{gamma_corollary}
If the assumptions required by either Theorem \ref{c_1thm} or \ref{del_thm} hold, then
\beq\label{scale_conv}
\lim_{N \to \infty } \gamma_N = 0
\eeq
and 
$x_N$ converges to a non-optimal point $\bar x$ such that
\beq\label{xbar} 
\bar x = [0, \bar x^{(2)}, \hdots, \bar x^{(n)}]^T.
\eeq
\end{corollary}
\begin{proof}
It is immediate from \eqref{sk_psi} or \eqref{sum_conv} that $|s^{(1)}_N| \to 0$ as $N\to\infty$, 
so from \eqref{gammadef}, we conclude \eqref{scale_conv}.  Also due to \eqref{sgnxk_1} we have  $|x^{(1)}_N| \to 0$, and since  $f(x_N) = a|x^{(1)}_{N }| + \sum^{n-1}_{i=2} x^{(i)}_N$ is bounded below, so is $\sum^{n-1}_{i=2} x^{(i)}_N$.
Due to \eqref{pos} and \eqref {dksk_1}, we have $d^{(i)}_{N-1} < 0$, for $i=2,3,\hdots,n$, so $t_{N-1}d^{(i)}_{N-1} = x^{(i)}_N - x^{(i)}_{N-1} < 0$, and therefore $x^{(i)}_N$ is strictly decreasing as  $N \to \infty$. Hence, $x^{(i)}_N$ converges to a limit $\bar x^{(i)}$.
\end{proof}

Due to the symmetry we discussed earlier, the total decrease along each component, $ x_0^{(i)} - \bar  x^{(i)} =  \sum_{k=0}^{N}s^{(i)}_k$, is the same for  $i=2,3,\hdots,n$.

Finally, note that it follows from Corollary \ref{gamma_corollary} together with \eqref{hkgen} that, when the assumptions hold, the matrix $H_N$ converges to zero. In contrast, when full BFGS is applied to the same problem, it is typically the case that a direction is identified along which $f$ is unbounded below within a few iterations, and that at the final iterate, one eigenvalue of the inverse Hessian is much smaller than the others, with its corresponding eigenvector close to the first coordinate vector, along which $f$ is nonsmooth.


\section{Experiments} \label{sec:expts}
Our experiments were conducted using the BFGS / L-BFGS \matlab\ code in 
{\sc hanso}.\footnote{www.cs.nyu.edu/overton/software/hanso/} 
This uses the Armijo-Wolfe bracketing line search given in Algorithm 2.
Consequently, according to the results of \S\ref{subsec:specificLS},
scaled memoryless BFGS (L-BFGS with $m=1$) should fail on function \eqref{fdef} when $a$ satisfies Assumption 1: $ 2\sqrt{n-1} \le a$. This is illustrated in
Figure \ref{fig:gl1}, which shows an experiment where we set $a=3$ and $n=2$ and ran scaled memoryless BFGS, the gradient method, and full BFGS, starting from the same randomly generated initial point. We see that scaled memoryless BFGS fails, in the sense that it converges
to a non-optimal point, while the gradient method succeeds, in the sense that it generates iterates with $f(x_k)\downarrow -\infty.$ In contrast to both, full BFGS succeeds in the sense
that it finds a direction along which $f$ is unbounded below in just five iterations. These three
different outcomes respectively illustrate the three different ways that the {\sc hanso} code terminated in our experiments: (i) convergence to a non-optimal point, which is detected when the steplenth upper bound $\beta$ in Algorithm 2
converges to zero indicating that Armijo-Wolfe points exist, but the line search terminates without finding one due to rounding errors; (ii) divergence of the $f(x_k)$ to $-\infty$ although the line search
always finds Armijo-Wolfe steplengths; and (iii) generation of a direction along which $f$ is
apparently unbounded below, which is detected when $\beta$ in Algorithm 2 remains equal to its initial value of $\infty$ while the lower bound $\alpha$ is repeatedly doubled until a limit is exceeded.\footnote{Although in
principle the code would alternatively terminate if a termination tolerance was met or an upper bound
on the number of iterations was exceeded, we set these so small and large respectively that
they virtually never caused termination.} In the results reported below for function \eqref{fdef}, termination (i) is considered a failure while terminations (ii) and (iii) are considered successes. We note that,
provided $\sqrt{n-1}\le a$, the gradient method can never result in termination (iii), and
whether it results in termination (i) or (ii) depends on the Armijo parameter \cite{AO18}. In our experiments, L-BFGS, with or without scaling and with one or more updates,
always resulted in termination (i) or (iii), while full BFGS 
invariably resulted in termination (iii) (as we know it must from the results in \cite{XW17}).

Although the proof of Theorem~\ref{conv} does require Assumption 1 
we observed that $ \sqrt{3(n-1)} \le a$ suffices for $\{|b_k|\}$ and consequently  $|d_k|/\|d_k\|_2$ to converge.
In Figure \ref{fig:gl2} we repeat the same experiment with $a=\sqrt{3}$ and $n=2$, showing that scaled memoryless BFGS still fails.
In this case, as noted in Section \ref{sec:theory}, the normalized direction is the same as the normalized direction generated by 
the gradient method, but unlike in the gradient method, the magnitude of the directions $d_k$ converge to zero
so scaled memoryless BFGS fails.

However, if we set $a$ to $\sqrt{3}-0.001$ the method succeeds. This is demonstrated in Figure \ref{fig:gl3}: observe
that although one at first has the impression that $x_k$ is converging to a non-optimal point, a search direction is generated
on which $f$ is unbounded below ``at the last minute".

\begin{figure} 
    \centering
    \includegraphics[scale=0.5]{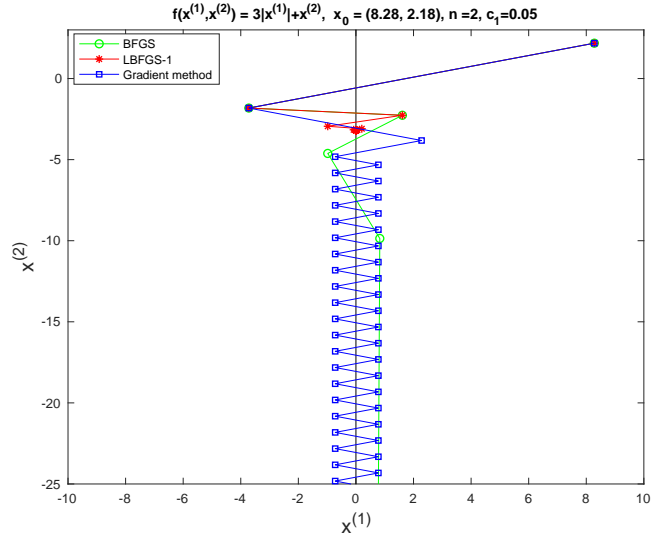} 
    \caption{Full BFGS (green circles), scaled memoryless BFGS (red asterisks) and the gradient method (blue squares) applied to the function \eqref{fdef} defined by $a=3$ and $n=2$. Scaled memoryless BFGS fails while full BFGS and the gradient method succeed.}
    \label{fig:gl1}
\end{figure}

\begin{figure} 
    \centering
    \includegraphics[scale=0.5]{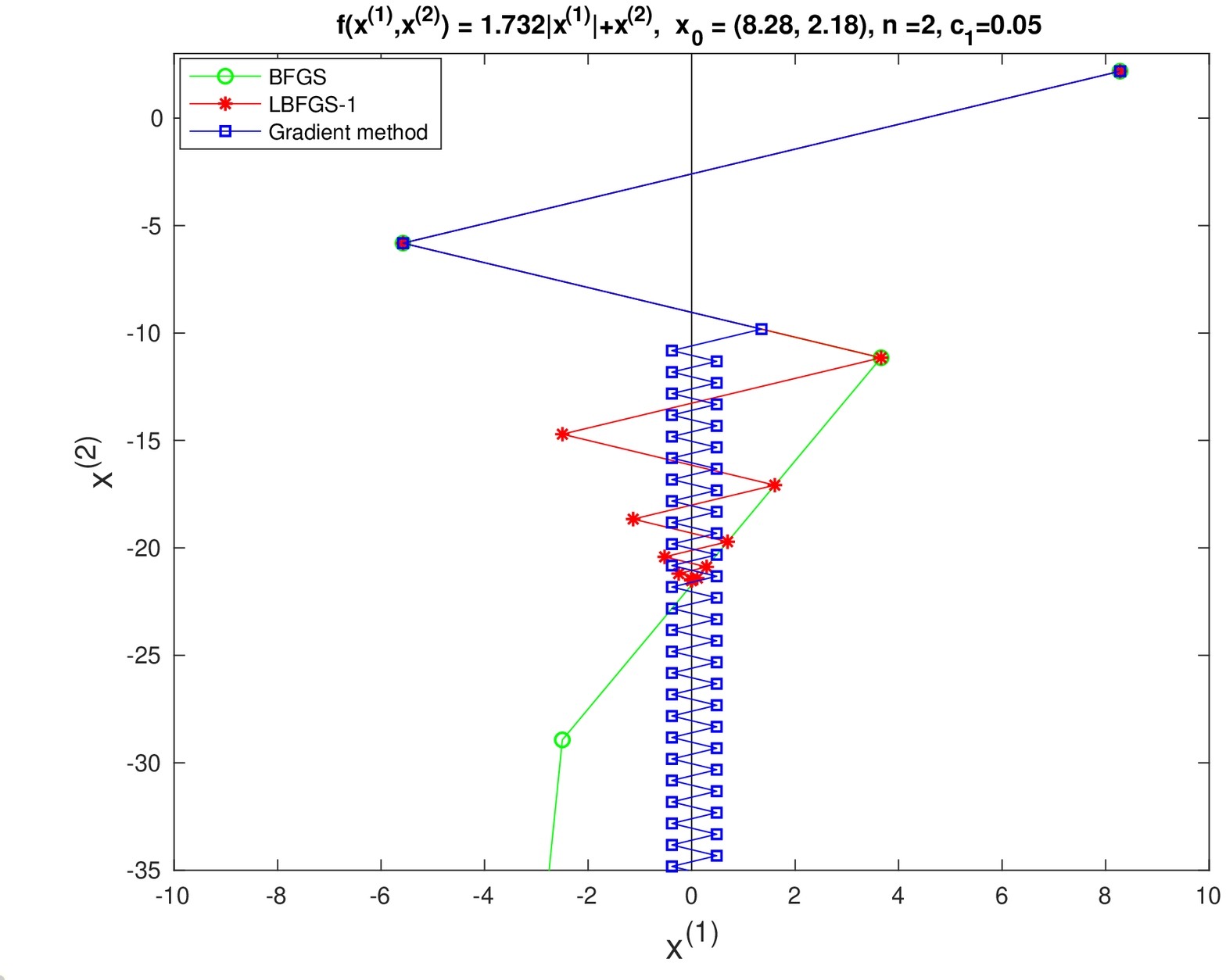} 
    \caption{Full BFGS (green circles), scaled memoryless BFGS (red  asterisks) and the gradient method (blue squares) applied to
    the function \eqref{fdef} defined by $a=\sqrt{3}$ and $n=2$.
   Scaled memoryless BFGS fails while full BFGS and the gradient method succeed.}
    \label{fig:gl2}
\end{figure}

\begin{figure} 
    \centering
    \includegraphics[scale=0.5]{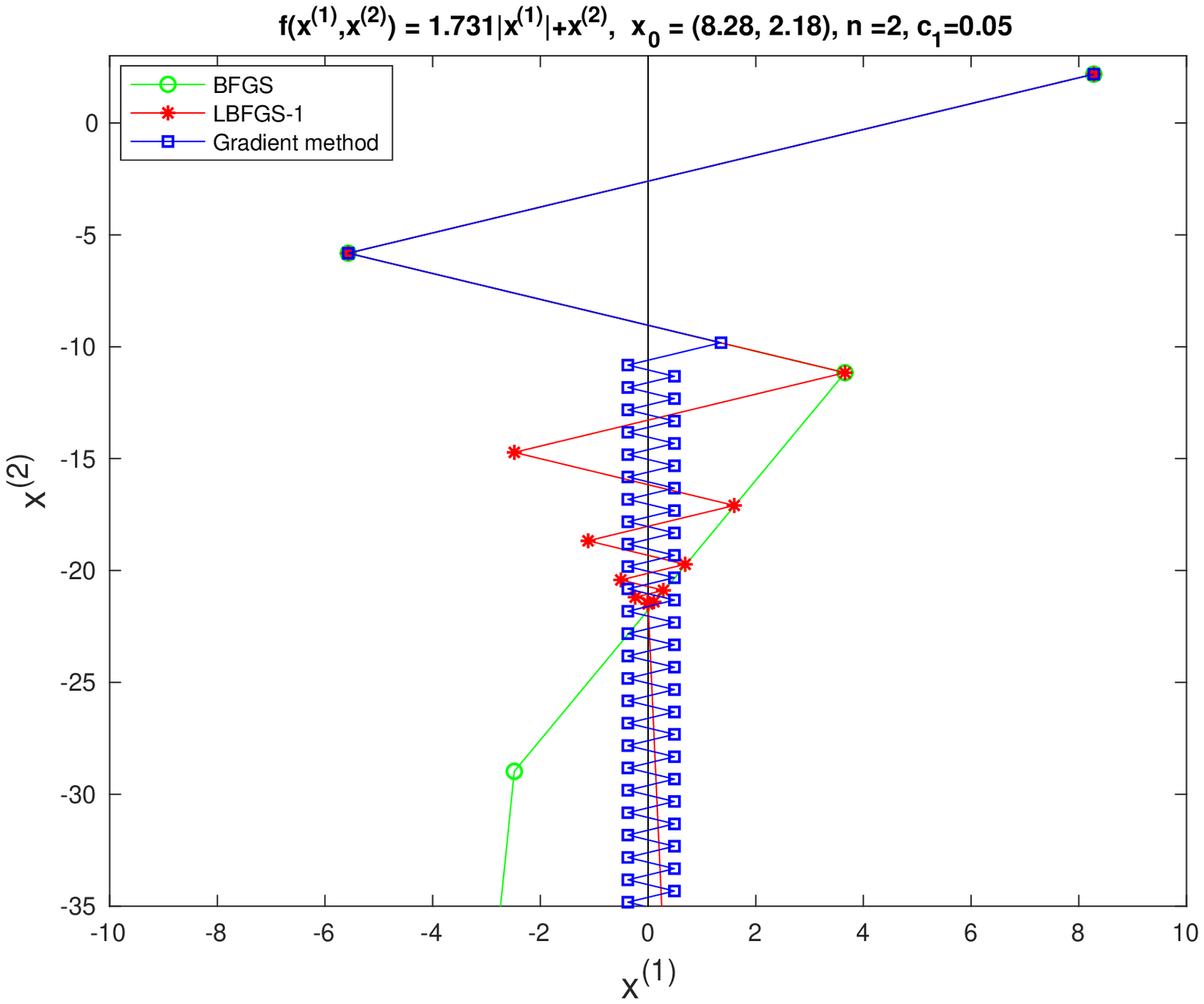} 
    \caption{Full BFGS (green circles), scaled memoryless BFGS (red  asterisks) and the gradient method (blue squares) applied to
    the function \eqref{fdef} defined by $a=\sqrt{3}-0.001$ and $n=2$.
    All methods succeed.}
    \label{fig:gl3}
\end{figure}


Extensive additional experiments verify that the condition $\sqrt{3(n-1)}\le a$, as opposed to Assumption~1, is sufficient for failure, as illustrated by the magenta asterisks in Figure \ref{fig:30_5k}. Starting from 5000 random points generated from the normal distribution, 
we called
scaled memoryless BFGS to minimize function \eqref{fdef} with $n=30$ and for values of $a$ ranging from $9.317$ to $9.337$,  
since for $n=30$,  $\sqrt{3(n-1)} \approx 9.327$. We see that for $9.327 \le a$  the failure rate is 1 (100\%), 
while for $9.32 > a$ the failure rate is 0.
In comparison to a similar experiment in \cite{AO18} for the gradient method, the transition from failure rate  0 to 1 is quite sharp here. 
This might be explained by the fact that the gradient method fails because the steplength $t_k\rightarrow 0$, whereas for scaled
memoryless BFGS, $t_k$ does not converge to zero; it is the scale $\gamma_k$ and consequently the norm of $d_k$ which converges to zero.
Hence, rounding error prevents the observation of a sharp transition in the results for the gradient method, as explained in \cite{AO18}; 
by comparison, rounding error plays a less significant role in the experiments reported here. 

The cyan squares in Figure \ref{fig:30_5k} show 
the results from the same experiment for memoryless BFGS \emph{without} scaling, i.e., with $H_k^0 = I$ instead of \eqref{Hk0def}, using
the same 5000 initial points. In this case, the method is successful regardless of the value of $a$. 
\begin{figure} 
    \centering
    \includegraphics[scale=0.55]{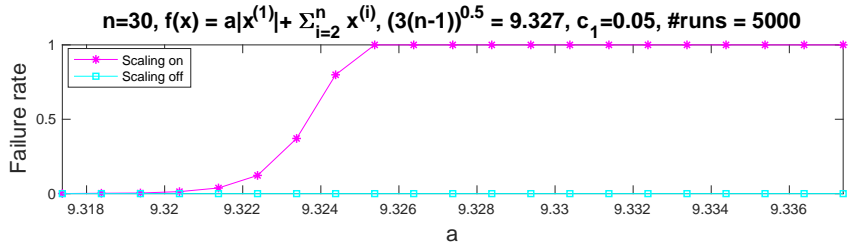} 
    \caption{The failure rate of memoryless BFGS with scaling (magenta asterisks) and without scaling (cyan squares) applied to function \eqref{fdef} with $n=30$ and 21 different values of $a$, initiating the method from  5000 random points. With scaling, the failure rate is 1  for $9.327 \le a$. Without scaling, the failure rate is 0 regardless of $a$.} 
    \label{fig:30_5k}
\end{figure}

 Experiments suggest that
 the theoretical results we presented for scaled L-BFGS with only one update might extend, although undoubtedly
 in a far more complicated form, to any number of updates. In Figure~\ref{fig:sdb1} we show results of experiments with a variety of choices of $m$ and $a$, running scaled L-BFGS-$m$ (L-BFGS with $m$ updates) initiated from 1000 randomly generated points for each pair ($m$,$a$). The horizontal axis shows $m$, the number of updates, while the vertical
 axis shows the observed failure rate.  We set the Armijo parameter $c_1=0.01$ and $n=4$, so that $\sqrt{3(n-1)} = 3$, and
 show results for values of $a$ ranging from 2.99 to 300. Figure \ref{fig:sdb2} shows results from the same experiment except that $c_1=0.001$.
The results shown in Figure \ref{fig:sdb3} use a different objective function; instead of \eqref{fdef}, we define $f(x) = a|b_{1}^Tx|+  b_{2}^Tx$, where $b_1$ and $b_2$ were each chosen as a random vector in $\R^{10}$ and normalized to have length one. The Armijo parameter was set to $c_1=0.01$.
 In all of Figures \ref{fig:sdb1}, \ref{fig:sdb2} and \ref{fig:sdb3} we observe that as $a$ gets larger for a fixed $m$, the failure rate increases. 
On the other hand, as $m$ gets larger for a fixed $a$, the failure rate decreases.
Comparing Figures \ref{fig:sdb1} and \ref{fig:sdb2}, we see that the results 
do not demonstrate a significant dependence  on the Armijo parameter $c_1$; in particular, 
as we established in Section \ref{subsec:specificLS},
there is no dependence on $c_1$
when $m=1$ because we are using the line search in Algorithm~2.
However, we do observe small differences for the larger values of $m$,
where the failure rate is slightly higher for the larger Armijo parameter. This is consistent with the theoretical
results in \S \ref{subsec:armijo} as well as those in \cite{AO18}, where, if $a$ is relatively large, then to avoid
failure $c_1$ should not be too large.
\begin{figure} 
	\centering
	\includegraphics[scale=0.5]{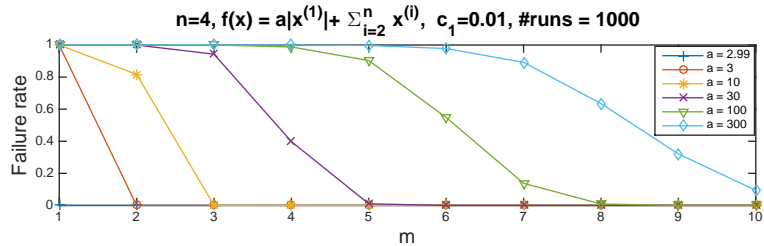}
	\caption{The failure rate for each scaled L-BFGS-$m$, where the number of updates
		$m$ ranges from 1 to 10, applied to function \eqref{fdef} with $a = 2.99$ (blue pluses), $a = 3$ (orange circles), $a = 10$ (yellow asterisks), $a = 30$ (purple crosses), $a = 100$ (green triangles) and finally $a = 300$ (cyan diamonds), with
		$c_1=0.01$ and $n=4$ and hence $\sqrt{3(n-1)} = 3$, and with
		each experiment initiated from  1000 random points.}
	\label{fig:sdb1}
\end{figure}
\begin{figure} 
	\centering
	\includegraphics[scale=0.5]{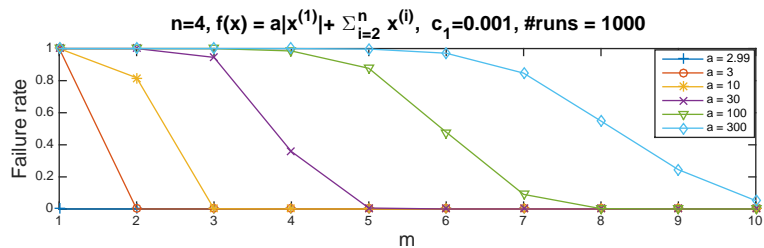} 
	\caption{
	The same experiment as in Figure \ref{fig:sdb1} except that $c_1=0.001$.}
	\label{fig:sdb2}
\end{figure}
\begin{figure} 
	\centering
	\includegraphics[scale=0.5]{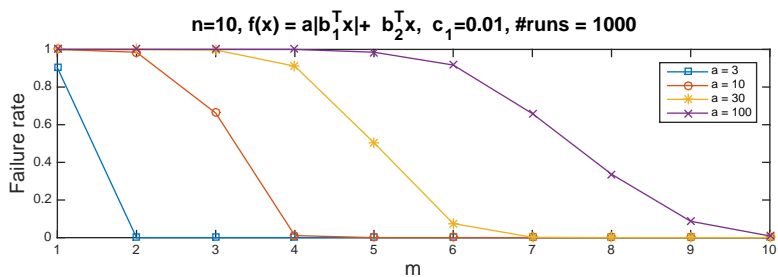}
	\caption{
	The same experiment as in Figure \ref{fig:sdb1} except that $f(x) = a|b_{1}^Tx|+  b_{2}^Tx$ where $b_1, b_2 \in\R^{10}$ were chosen randomly.}
	\label{fig:sdb3}
\end{figure}

Finally, we conducted experiments with a more general class of piecewise linear max functions defined as
	\beq\label{pwl}
	f(x) = \max_{i=1,\ldots p}~ \{b_i^T x - r_i\},
	\eeq
	where $b_1,...,b_p$ are randomly generated vectors in $\R^n$ and $r_1,...,r_p$ are random scalars. These quantities were fixed for the experiment reported here but similar results were obtained for other choices. 
We set $n=10$ and $p=50$, obtaining a problem that, unlike those studied above, is bounded below.  Consequently, all runs result in termination (i), and we evaluated how successful they were by comparing the final function
value to the optimal value $f_*$ that we obtained via linear programming using {\sc mosek}\footnote{https://www.mosek.com/} with the tolerance set to $10^{-14}$.  Figure \ref{fig:pwl2} shows the median accuracy obtained by L-BFGS-$m$, for $m=1,\hdots,10$, with and without scaling. L-BFGS with scaling does not achieve a median accuracy better than $10^{-2}$, even when $m=10$. Without scaling, the accuracy of the results improves substantially, to a median accuracy of about $10^{-9}$ with $m=9$. Strangely, for this problem, and many different instances of it that we tried, L-BFGS-10 performs worse than L-BFGS-9. The median accuracy of the solution found by full BFGS (with or without scaling the initial inverse Hessian approximation) is significantly better: about $10^{-14}$. 

\begin{figure} 
	\centering
	\includegraphics[scale=0.5]{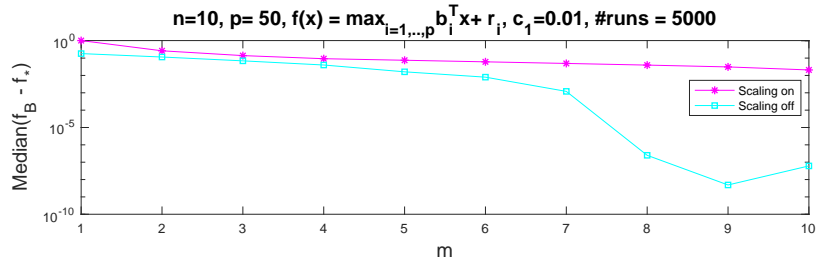} 
	\caption{
	Median accuracy of the solution $f_B$ found by L-BFGS-$m$ with $m=1,\hdots,10$ for the piecewise linear function defined in \eqref{pwl}, with $n=10$ and $p=50$, compared with the value $f_*$ obtained from the linear optimizer in {\sc mosek} using high accuracy. Scaled L-BFGS-$m$ does not obtain accurate solutions even with $m=10$. In contrast, with scaling off, L-BFGS-9 obtains a median accuracy of about $10^{-9}.$}
	\label{fig:pwl2} 
\end{figure}

\section{Concluding Remarks}\label{sec:conclude}

We have given the first analysis of a variant of L-BFGS applied to a nonsmooth function, showing that the scaled version of
memoryless BFGS (L-BFGS with just one update) applied to \eqref{fdef}
generates iterates converging to a non-optimal point under simple conditions.
One of these conditions applies to the method with any Armijo-Wolfe line search and depends on the Armijo parameter. The
other condition applies to the method using a standard Armijo-Wolfe bracketing line search and does not depend on the Armijo
parameter. Experiments suggest that extended results likely hold for L-BFGS with more than one update, though clearly a generalized analysis would be much more complicated.

We do not know
 whether L-BFGS without scaling applied to the same function can converge to a non-optimal point,
but numerical experiments suggest that this cannot happen.
Furthermore, we observed that L-BFGS without scaling obtains significantly more accurate solutions than L-BFGS with scaling when applied to a more general piecewise linear function that is bounded below. Nonetheless, it remains an open question as to whether scaling is generally inadvisable when applying L-BFGS to 
nonsmooth functions, despite its apparent advantage for smooth optimization.

{\bf Acknowledgments.} Many thanks to Margaret H.~Wright for arranging financial support 
for the first author from the Simons Foundation. Thanks also to the anonymous referees for
carefully reading the paper and suggesting several improvements.

\appendix 
\section{Proof of Lemma \ref{lemma1}} \label{appendA}
Suppose $\sqrt{3(n-1)} \le a$. Using a change of variable such that $\beta_k = b_k$ when $k$ is even, and $\beta_k = -b_k$ when $k$ is odd, \eqref{bk} becomes 
\beq\label{anbk}
\beta_k =  \dfrac{1+ (n-1)\beta_{k-1}^2}{a-(n-1)\beta_{k-1}} -\beta_{k-1}.
\eeq

From \eqref{b0} we have $\beta_0 = 1/a$. Using induction we prove that $0 <\beta_k \le 1/a$. This is clearly true for $k=0$. 
Suppose we have $0 <\beta_{k-1} \le 1/a $. Hence
$$ \beta_{k-1} < \dfrac{1}{a-(n-1)\beta_{k-1}}< \dfrac{1+(n-1)\beta_{k-1}^2}{a-(n-1)\beta_{k-1}}, $$
so, dropping the middle term and moving $\beta_{k-1}$ to the R.H.S., we get exactly the definition of $\beta_k$ according to \eqref{anbk}. So, we have $0 < \beta_k$. Next, starting from  $\sqrt{3(n-1)} \le a$, we show that $\beta_k \le 1/a$:
\begin{align*}
&  \dfrac{3(n-1)}{a} \le a \RA   \\
& \dfrac{(n-1)}{a} +2(n-1)\beta_{k-1} \le a \RA \\
&\dfrac{a^2+n-1}{a} \le 2(a - (n-1)\beta_{k-1}) \RA \\
&\dfrac{a^2+n-1}{a(a - (n-1)\beta_{k-1})} \le 2.
\end{align*}
Multiplying both sides by $\beta_{k-1}$ we get
$$\dfrac{a\beta_{k-1}+1}{a - (n-1)\beta_{k-1}} -\dfrac{1}{a} \le 2\beta_{k-1},$$
and finally by moving $1/a$ to the right and $2\beta_{k-1}$ to the left we get
$$ \dfrac{1+(n-1)\beta_{k-1}^2}{a - (n-1)\beta_{k-1}} -\beta_{k-1} \le \dfrac{1}{a}.$$
The L.H.S.\ is $\beta_k $ as it's defined in \eqref{anbk}, so $ \beta_k \le 1/a$. Recalling the change of variable in the beginning of the proof it follows that $\beta_k = |b_k|$. So, from \eqref{anbk} we get \eqref{abk}.

\section{Proof of Theorem \ref{conv}}
\label{appendB}
We continue to use the same change of variable as before, that is $\beta_k = b_k$ when $k$ is even, and $\beta_k = -b_k$ when $k$ is odd. In this way, \eqref{anbk} is equivalent to \eqref{abk}, and we prove that if $ 2\sqrt{n-1} \le a$, then
 $\{\beta_k\}$ converges. 
From a little rearrangement in \eqref{anbk} we can easily get
\beq \label{eq1}
a(\beta_k+\beta_{k-1}) =   1+ 2(n-1)\beta_{k-1}^2 + (n-1)\beta_{k-1}\beta_k ,  
\eeq
and by moving $(n-1)\beta_{k-1}\beta_k$ to the left and adding 1 to both sides we get
\beq \label{eq2}
 a(\beta_k+\beta_{k-1})-(n-1)\beta_{k-1}\beta_k+1 =  2\Big(1+(n-1)\beta_{k-1}^2\Big) .  
\eeq
For further simplification we define  

\beq\label{rho}
\rho_k= \dfrac{1+(n-1)\beta_k^2}{a-(n-1)\beta_k},
\eeq
so we can rewrite \eqref{anbk} as
\beq\label{beta_rho}
 \beta_{k+1} = \rho_k -  \beta_k.
\eeq
By applying  \eqref{beta_rho} recursively we obtain
\beq\label{k2k}
\beta_{k+1} - \beta_{k-1}=   \rho_{k} - \rho_{k-1}.
\eeq
Note that from \eqref{rho} we have
\begin{align*}
&\rho_k - \rho_{k-1}  =\dfrac{1+(n-1)\beta_k^2 }{a-(n-1)\beta_k}  -\dfrac{ 1+(n-1)\beta_{k-1}^2}{a-(n-1)\beta_{k-1}}\\
&=\dfrac{\Big( 1+(n-1)\beta_k^2 \Big)\Big(a-(n-1)\beta_{k-1}\Big)  -\Big( 1+(n-1)\beta_{k-1}^2 \Big)\Big(a-(n-1)\beta_{k}\Big)  }   {\Big(a-(n-1)\beta_k\Big)\Big(a-(n-1)\beta_{k-1}\Big)}\\
&=\dfrac{(\beta_k-\beta_{k-1})(n-1)\Big(a(\beta_k+\beta_{k-1}) -(n-1)\beta_{k-1}\beta_k +1\Big) }{\Big(a-(n-1)\beta_k\Big)\Big(a-(n-1)\beta_{k-1}\Big)}. \numberthis \label{eqPhi}  
\end{align*} 
The last factor in the numerator is the L.H.S.\  in \eqref{eq2}, so
\beq
\rho_k - \rho_{k-1}  = \dfrac{(\beta_k-\beta_{k-1})(n-1)2\Big(1+ (n-1)\beta_{k-1}^2\Big) }{\Big(a-(n-1)\beta_k\Big)\Big(a-(n-1)\beta_{k-1}\Big)}.
\eeq
Hence, since all of the factors in this product except $(\beta_{k}-\beta_{k-1})$ are known to be positive, we have
\beq\label{mon}
(\rho_{k}-\rho_{k-1})(\beta_{k}-\beta_{k-1})\ge 0.
\eeq
Putting \eqref{k2k} and \eqref{mon} together we conclude 
\beq\label{inc}
( \beta_{k+1} - \beta_{k-1})( \beta_{k}-\beta_{k-1} ) \ge 0.
\eeq

As the next step we will show that 
\beq\label{inc_dec}
( \beta_{k+1} - \beta_{k})( \beta_{k}-\beta_{k-1} ) \le 0.
\eeq
Since $a \ge 2\sqrt{n-1}$ and using $1/a \ge \beta_{k-1}$ we get 
\begin{align*}
&\Big(a^2 -4(n-1)\Big) \Big(a^2 +(n-1)\Big)\ge 0 \RA\\
& a^2 -3(n-1) \ge\dfrac{4(n-1)^2}{a^2}\RA\\
&a^2 -3(n-1) \ge 4(n-1)^2\beta_{k-1}^2 \RA\\
&a^2 -3(n-1) -4(n-1)^2\beta_{k-1}^2 \ge 0.
\end{align*}
By adding and deducting $2(n-1)^2\beta_k\beta_{k-1}$ to the L.H.S.\  above we get
\[
a^2-2(n-1)\Big(1+2(n-1)\beta_{k-1}^2+(n-1)\beta_{k-1}\beta_k\Big) +2(n-1)^2\beta_k\beta_{k-1} -(n-1)\ge 0.
\]  
By combining this with  \eqref{eq1}  we get
$$ a^2-2(n-1)a(\beta_k+\beta_{k-1}) +2(n-1)^2\beta_k\beta_{k-1} -(n-1)\ge 0.  $$
By moving some of the terms to the R.H.S.\ and factorizing the L.H.S.\  we get
$$
\Big(a-(n-1)\beta_k\Big)\Big(a-(n-1)\beta_{k-1}\Big) \ge  a(n-1)(\beta_k+\beta_{k-1}) -(n-1)^2\beta_k\beta_{k-1} +(n-1),
$$
which we can write as  
\beq \label{ineq2}
1 \ge  \dfrac{(n-1)\Big(a(\beta_k+\beta_{k-1}) -(n-1)\beta_k\beta_{k-1} +1\Big)}{\Big(a-(n-1)\beta_k\Big)\Big(a-(n-1)\beta_{k-1}\Big)}.
\eeq
Now, suppose $\beta_k -\beta_{k-1} \ge 0$. Multiplying both sides of the inequality \eqref{ineq2} by $\beta_k -\beta_{k-1}$, according to \eqref{eqPhi} we get
$$ \beta_k-\beta_{k-1}  \ge \rho_k - \rho_{k-1},$$
so, 
\[
\rho_{k-1} -\beta_{k-1}  \ge \rho_k - \beta_k
\]
which means that via \eqref{beta_rho} we have shown $ \beta_k \ge \beta_{k+1}$. Alternatively, if we had $\beta_k -\beta_{k-1} \le 0$ above, then we would get $ \beta_k \le \beta_{k+1}$. Hence, we always have $(\beta_{k+1}-\beta_k)(\beta_{k} -\beta_{k-1}) \le 0$, which is exactly inequality \eqref{inc_dec}.

Since  we start with $\beta_0 = 1/a$, according to Lemma \ref{lemma1} we have $\beta_1 \le \beta_0$. Using \eqref{inc_dec} inductively we get 
\[\beta_1 - \beta_0 \le 0,~~ 0 \le \beta_2-\beta_1,~~ \beta_3 - \beta_2 \le 0, \hdots\]
and from applying \eqref{inc} to each one of these inequalities we conclude
\[\beta_2 - \beta_0 \le 0,~~ 0 \le \beta_3-\beta_1,~~ \beta_4 - \beta_2 \le 0, \hdots\]
which shows that we can split $\{\beta_k\}$ into two separate monotonically decreasing and  increasing subsequences:
\begin{align*} 
0 &< \hdots \beta_4 \le \beta_2 \le \beta_0=1/a, \\
0 &< \beta_1 \le \beta_3 \le \beta_5 \hdots <1/a.
\end{align*}
By the bounded monotone convergence theorem we conclude that each one of these subsequences converge, i.e.\\
\[\lim_{k \to \infty} |\beta_{k+2} - \beta_k | = 0, \]
and recalling \eqref{k2k} we get
\[\lim_{k \to \infty} |\rho_{k+1} - \rho_k | = 0. \]
On the other hand, looking at the equality in \eqref{eqPhi} we know that except $(\beta_{k+1} - \beta_k)$ all the 
factors in the numerator and denominator are bounded away from zero. So therefore we must have 
\[\lim_{k \to \infty} |\beta_{k+1} - \beta_k | = 0, \]
and hence, since the even and odd sequences both converge, they must have the same limit.
Using the definition of $\beta_{k+1}$ in \eqref{anbk} we get 
\[ \lim_{k \to \infty}  \left|\dfrac{1+ (n-1)\beta_k^2}{a-(n-1)\beta_k} -2\beta_k \right|  = 0.
\] 
Since the denominator is bounded away from zero we must have
\[ \lim_{k \to \infty}  3(n-1)\beta_k^2 -2a\beta_k + 1  = 0.
\] 
The two roots of the limiting quadratic equation are $$\dfrac{a \pm \sqrt{a^2-3(n-1)}}{3(n-1)}.$$ 
The smaller root is $b$ as defined in \eqref{b} and the larger root 
is greater than $1/a $, which according to Lemma \ref{lemma1} is not possible. Hence, 
$$ \lim_{k \to \infty} \beta_k =  \lim_{k \to \infty} |b_k|= b. $$

\bibliographystyle{alpha}
\bibliography{refs_mo}

\newcommand{\etalchar}[1]{$^{#1}$}
\begin{thebibliography}{LNC{\etalchar{+}}11}

\bibitem[AO18]{AO18}
Azam Asl and Michael~L. Overton.
\newblock {Analysis of the Gradient Method with an Armijo-Wolfe Line Search on
  a Class of Nonsmooth Convex Functions}.
\newblock September 2018.
\newblock arXiv:1711.08517v2.

\bibitem[Cla90]{Cla90}
F.~H. Clarke.
\newblock {\em Optimization and nonsmooth analysis}, volume~5 of {\em Classics
  in Applied Mathematics}.
\newblock Society for Industrial and Applied Mathematics (SIAM), Philadelphia,
  PA, second edition, 1990.

\bibitem[CMO17]{CurMitOve17}
Frank~E. Curtis, Tim Mitchell, and Michael~L. Overton.
\newblock A {BFGS}-{SQP} method for nonsmooth, nonconvex, constrained
  optimization and its evaluation using relative minimization profiles.
\newblock {\em Optim. Methods Softw.}, 32(1):148--181, 2017.

\bibitem[Dai02]{DAI02}
Yu-Hong Dai.
\newblock Convergence properties of the {BFGS} algorithm.
\newblock {\em SIAM J. Optim.}, 13(3):693--701 (2003), 2002.

\bibitem[GL03]{rh03}
Philip~E. Gill and Michael~W. Leonard.
\newblock Limited-memory reduced-hessian methods for large-scale unconstrained
  optimization.
\newblock {\em {SIAM} Journal on Optimization}, 14(2):380--401, 2003.

\bibitem[GL18]{GL18}
J.~Guo and A.~Lewis.
\newblock Nonsmooth variants of {P}owell's {BFGS} convergence theorem.
\newblock {\em SIAM Journal on Optimization}, 28(2):1301--1311, 2018.

\bibitem[LMH16]{HarchaouiXX}
Hongzhou {Lin}, Julien {Mairal}, and Zaid {Harchaoui}.
\newblock {An Inexact Variable Metric Proximal Point Algorithm for Generic
  Quasi-Newton Acceleration}.
\newblock {\em arXiv e-prints}, page arXiv:1610.00960, Oct 2016.

\bibitem[LN89]{LN89}
Dong~C. Liu and Jorge Nocedal.
\newblock On the limited memory {BFGS} method for large scale optimization.
\newblock {\em Math. Programming}, 45(3, (Ser. B)):503--528, 1989.

\bibitem[LNC{\etalchar{+}}11]{DL11}
Quoc~V. Le, Jiquan Ngiam, Adam Coates, Abhik Lahiri, Bobby Prochnow, and
  Andrew~Y. Ng.
\newblock On optimization methods for deep learning.
\newblock In {\em Proceedings of the 28th International Conference on
  International Conference on Machine Learning}, ICML'11, pages 265--272, USA,
  2011. Omnipress.

\bibitem[LO13]{LO13}
Adrian~S. Lewis and Michael~L. Overton.
\newblock Nonsmooth optimization via quasi-{N}ewton methods.
\newblock {\em Math. Program.}, 141(1-2, Ser. A):135--163, 2013.

\bibitem[LZ15]{LZ15}
A.~S. Lewis and S.~Zhang.
\newblock Nonsmoothness and a variable metric method.
\newblock {\em J. Optim. Theory Appl.}, 165(1):151--171, 2015.

\bibitem[Mas04]{MAS04}
Walter~F. Mascarenhas.
\newblock The {BFGS} method with exact line searches fails for non-convex
  objective functions.
\newblock {\em Math. Program.}, 99(1, Ser. A):49--61, 2004.

\bibitem[MR15]{AM15}
Aryan Mokhtari and Alejandro Ribeiro.
\newblock Global convergence of online limited memory bfgs.
\newblock 16:3151--3181, 2015.

\bibitem[NW06]{NW06}
J.~Nocedal and S.~J. Wright.
\newblock {\em Numerical Optimization}.
\newblock Springer, New York, 2nd edition, 2006.

\bibitem[Pow76]{POW76b}
M.~J.~D. Powell.
\newblock Some global convergence properties of a variable metric algorithm for
  minimization without exact line searches.
\newblock In {\em Nonlinear Programming}, pages 53--72, Providence, 1976. Amer.
  Math. Soc.
\newblock SIAM-AMS Proc., Vol. IX.

\bibitem[XW17]{XW17}
Yuchen Xie and Andreas Waechter.
\newblock {On the convergence of BFGS on a class of piecewise linear non-smooth
  functions}.
\newblock December 2017.
\newblock arXiv:1712.08571.

\end{thebibliography}
\end{document}